\documentclass[reqno,11pt]{amsart}
\usepackage{bm,url,amsmath,amssymb,amsthm}

\usepackage{mathrsfs,stmaryrd}
\usepackage{enumerate}

\usepackage{graphicx,epic,eepic}
\usepackage[percent]{overpic}
\usepackage[all]{xy}

\usepackage[top=1in, bottom=1in, left=1in, right=1in]{geometry}
\usepackage{setspace}
\allowdisplaybreaks

\usepackage[titletoc,title]{appendix}

\newcommand{\A}{\mathbb{A}}
\newcommand{\C}{\mathbb{C}}
\newcommand{\Hp}{\mathbb{H}}
\newcommand{\N}{\mathbb{N}}

\newcommand{\R}{\mathbb{R}}
\newcommand{\Z}{\mathbb{Z}}
\newcommand{\rar}{\rightarrow}
\newcommand{\bsl}{\backslash}
\newcommand{\Dbar}{\overline{D}}
\newcommand{\SL}{\mathrm{SL}}
\newcommand{\PSL}{\mathrm{PSL}}
\newcommand{\rrbr}{\rrbracket}
\newcommand{\llbr}{\llbracket}

\DeclareMathOperator{\Rl}{Re}
\DeclareMathOperator{\Ig}{Im}
\DeclareMathOperator{\per}{per}
\DeclareMathOperator{\Tr}{Tr}
\DeclareMathOperator{\sign}{sign}
\DeclareMathOperator{\spec}{spec}

\newtheorem{theorem}{Theorem}[section]
\newtheorem{corollary}[theorem]{Corollary}
\newtheorem{proposition}[theorem]{Proposition}
\newtheorem{lemma}[theorem]{Lemma}

\theoremstyle{definition}
\newtheorem{remark}[theorem]{Remark}
\newtheorem*{acknowledgements}{Acknowledgements}

\numberwithin{equation}{section}

\title{Distribution of the periodic points of the Farey map}
\author{Byron Heersink}
\address{Department of Mathematics, The Ohio State University, Columbus, OH 43210}
\email{heersink.5@osu.edu}

\dedicatory{With an Appendix by Florin P. Boca, Byron Heersink, and Claire Merriman}

\begin{document}

\begin{abstract}
We expand the cross section of the geodesic flow in the tangent bundle of the modular surface given by Series to produce another section whose return map under the geodesic flow is a double cover of the natural extension of the Farey map. We use this cross section to extend the correspondence between the closed geodesics on the modular surface and the periodic points of the Gauss map to include the periodic points of the Farey map. Then, analogous to work of Pollicott, we prove an equidistribution result for the periodic points of the Farey map when they are ordered according to the length of their corresponding closed geodesics.
\end{abstract}

\maketitle

\section{Introduction}

The work of Series \cite{S} provided a lucid explanation of the connection, first noticed by Artin \cite{Art}, between continued fractions and the geodesics in the modular surface. This relationship gives rise particularly to a correspondence between the periodic orbits of the Gauss map and the primitive closed geodesics. Utilizing this connection and the thermodynamic formalism for the Gauss map due to Mayer \cite{M}, Pollicott \cite{P} proved that the periodic points of the Gauss map, i.e., the periodic continued fractions and equivalently the reduced quadratic irrationals, become equidistributed on the unit interval with respect to the Gauss measure when ordered according to the length of their corresponding closed geodesics. 

The goal of this paper is expand the work of Series and Pollicott to encompass the Farey map, which is a slow down of the Gauss map. We first enlarge Series' cross section of the geodesic flow in the tangent bundle to the modular surface that, together with its first return map under the geodesic flow, forms a double cover of the natural extension of the Gauss map, to yield a cross section forming a double cover of the natural extension of the Farey map. This allows us to naturally extend the correspondence between closed geodesics and reduced quadratic irrationals to include the periodic points of the Farey map. We then utilize the work of Pollicott to establish the equidistribution of the periodic points of the Farey map, and of its natural extension, according to their invariant measures.

In addition to the equidistribution of the reduced quadratic irrationals, Pollicott \cite{P} proved the equidistribution of the closed geodesics in the modular surface. The technique he used was to establish the domain of an appropriate dynamical zeta function constructed from determinants of the Ruelle-Perron-Frobenius operator of the Gauss map. Similar applications of this idea include the earlier work of Parry and Pollicott \cite{PP} and Parry \cite{Pa} on the counting and equidistribution of orbits of Axiom A flows, and the more recent work of Kelmer \cite{Ke} on the equidistribution of closed geodesics in the modular surface with a specified linking number, as well as their corresponding reduced quadratic irrationals. Additional applications involving the distribution of periodic points of certain expanding maps of the unit interval include \cite{KN,Mo}. See also \cite{L} for an account of these types of results involving more probabilistic methods.

What is interesting about our results is that the Farey map and its natural extension have infinite invariant measures, which raises some difficulties in proving the equidistribution of periodic points. We mitigate these difficulties by introducing appropriate weights for the points which counteract the infinite measures, and formulate our results as  the equidistribution of weighted periodic points according to the finite Lebesgue measure. This, in addition to the use of inducing to obtain the Gauss map from the Farey map, then allows us to use Pollicott's analysis of the Ruelle-Perron-Frobenius operator of the Gauss map to obtain our results. A related equidistribution result for the periodic points of the Farey map has been independently derived by Pollicott and Urba\'nski as an application of their recent work in conformal graph directed Markov systems \cite{PU}. Their general results can also be applied to other parabolic systems, most notably to counting problems in Apollonian circle packings (see \cite{KO,OS1,OS3,OS2}).

In Section \ref{Sec2}, we cover some of the basic properties of continued fractions. We also define the Gauss and Farey maps and their natural extensions, characterize their periodic points, and formulate our main equidistribution results. In Section \ref{Sec3}, we review the connection between the modular surface and continued fractions. In particular, we recall Series' cross section of the geodesic flow, which we enlarge to yield another section whose return map under the geodesic flow is a double cover of the Farey map's natural extension. We then use this new section to extend the correspondence between closed geodesics in the modular surface and the periodic points of the Gauss map to those of the Farey map. In Section \ref{Sec4}, we prove our main equidistribution result utilizing the relationship between the periodic points of the Farey and Gauss maps to essentially reduce the problem to proving the equidistribution of the Gauss periodic points over certain continuous functions on $(0,1]$ which are allowed to have a vertical asymptote at $0$. We thus adapt Pollicott's work on the Ruelle-Perron-Frobenius operator of the Gauss map, being careful to account for a possible asymptote in a function used to define the operator. The aforementioned results have appeared in the author's Ph.D.\ thesis \cite{H}. The appendix, which is joint with Florin P.\ Boca and Claire Merriman, additionally proves a variation of the equidistribution results with explicit error terms using the ideas of Kallies et al.\ \cite{KOPS}, Boca \cite{Bo}, and Ustinov \cite{Ust}.

\section{Continued fractions and the Gauss and Farey maps}\label{Sec2}

\subsection{Continued fractions}\label{ssec:cfracs}
Throughout this paper, we are concerned with the regular continued fractions of the form
\[[a_1,a_2,\ldots]:=\cfrac{1}{a_1+\cfrac{1}{a_2+\ddots}},\qquad(a_j\in\N)\]
and we also make use of the notation
\[[a_0;a_1,a_2,\ldots]:=a_0+[a_1,a_2,\ldots].\qquad(a_0\in\Z,a_j\in\N)\]

For a given sequence $a=(a_j)_{j=1}^\infty$ of positive integers, define the nonnegative, coprime integers $p_n=p_n(a)=p_n(a_1,\ldots,a_n)$, $q_n=q_n(a)=q_n(a_1,\ldots,a_n)$ by
\[\frac{p_n}{q_n}:=[a_1,a_2,\ldots,a_n].\]
Denoting also $p_0=p_0(a):=0$ and $q_0=q_0(a):=1$, one can find that for $n\geq2$, $p_n=a_np_{n-1}+p_{n-2}$ and $q_n=a_nq_{n-1}+q_{n-2}$, which imply
\begin{equation}\label{CFmatrices}
\left(\begin{array}{cc}a_n&1\\1&0\end{array}\right)\left(\begin{array}{cc}a_{n-1}&1\\1&0\end{array}\right)\cdots\left(\begin{array}{cc}a_1&1\\1&0\end{array}\right)=\left(\begin{array}{cc}q_n&p_n\\q_{n-1}&p_{n-1}\end{array}\right).
\end{equation}
This in turn gives $p_{n-1}q_n-p_nq_{n-1}=(-1)^n$, and hence $\frac{p_{n-1}}{q_{n-1}}-\frac{p_n}{q_n}=\frac{(-1)^n}{q_{n-1}q_n}$. Also, for $x\geq0$, we have
\begin{equation}\label{CFprop1}
[a_1,\ldots,a_{n-1},a_n+x]=\frac{p_n+p_{n-1}x}{q_n+q_{n-1}x}.
\end{equation}

For $j,n\in\N$ with $j\leq n$, we define the positive coprime integers $p_{j,n}=p_{j,n}(a)$ and $q_{j,n}=q_{j,n}(a)$  by
\[\frac{p_{j,n}}{q_{j,n}}:=[a_j,a_{j+1},\ldots,a_n].\]
Taking transposes in \eqref{CFmatrices} reveals that $q_n(a_1,\ldots,a_n)=q_n(a_n,\ldots,a_1)$, which then implies that $q_n(a)=a_1q_{2,n}(a)+q_{3,n}(a)$, and more generally, $q_{j,n}(a)=a_jq_{j+1,n}(a)+q_{j+2,n}(a)$ for $j\leq n-2$. This equality extends to $j=n-1,n$ once we denote $q_{n+1,n}=q_{n+1,n}(a):=1$ and $q_{n+2,n}=q_{n+2,n}(a):=0$. Another subtle property of the values $q_{j,n}$ we wish to note is that
\begin{equation}\label{prodprop}
\prod_{j=1}^n\frac{p_{j,n}}{q_{j,n}}=\frac{1}{q_n}.
\end{equation}
(See \cite[Lemma 2.1, Theorem 3.6]{F} for proof.) Taking transposes in \eqref{CFmatrices} furthermore reveals that
\begin{equation}\label{CFinvert}
\frac{q_{n-1}}{q_n}=[a_n,a_{n-1},\ldots,a_1].
\end{equation}

Lastly, we define for a finite tuple $b=(b_1,\ldots,b_n)\in\N^n$ the set
\[I_b=\llbr b_1,\ldots,b_n\rrbr:={}\{[b_1,\ldots,b_n+t]:t\in[0,1]\}
=\{[a_1,a_2,\ldots]\in[0,1]:a_j=b_j,j=1,\ldots n\},\]
which is the closed interval between $\frac{p_{n+1}(b,1)}{q_{n+1}(b,1)}$ and $\frac{p_n(b)}{q_n(b)}$. We thus have
\begin{equation}
m(I_b)=\left|\frac{p_{n+1}(b,1)}{q_{n+1}(b,1)}-\frac{p_n(b,1)}{q_n(b,1)}\right|=\frac{1}{q_{n+1}(b,1)q_n(b,1)}
=\frac{1}{q_n(b)(q_n(b)+q_{n-1}(b))}.\label{cfgapmeas}
\end{equation}
Here and throughout this paper, $m$ denotes the Lebesgue measure on $[0,1]$.

\subsection{The Gauss map}\label{ssec:Gmap}
The Gauss map $G:[0,1]\rar[0,1]$ is defined by
\[G(x):=\begin{cases}
\{x^{-1}\}&\text{if $x\neq0$}\\
0&\text{if $x=0$}
\end{cases}\]
where $\{x\}=x-\lfloor x\rfloor$ denotes the fractional part. This map is invariant with respect to the Gauss measure $\nu$ given by
\[d\nu:=\frac{dx}{(1+x)\log2}.\]
We define the natural extension $\tilde{G}:[0,1]^2\rar[0,1]^2$ of the Gauss map by
\[\tilde{G}(x,y):=\left(G(x),\frac{1}{\lfloor x^{-1}\rfloor+y}\right)\]
which is invariant with respect to the measure $\tilde{\nu}$ given by
\[d\tilde{\nu}:=\frac{dx\,dy}{(1+xy)^2\log2}.\]
The action of $G$ and $\tilde{G}$ on continued fractions is as follows:
\begin{align*}
G([a_1,a_2,\ldots])&=[a_2,a_3,\ldots];\\
\tilde{G}([a_1,a_2,\ldots],[b_1,b_2,\ldots])&=([a_2,a_3,\ldots],[a_1,b_1,b_2,\ldots]).
\end{align*}
In other words, $G$ and $\tilde{G}$ act respectively as the one and two-sided shifts on the continued fraction expansions of their arguments. From these equalities, it is easy to see that the periodic points of $G$ are exactly the periodic continued fractions of the form
\[[\overline{a_1,\ldots,a_n}]:=[a_1,\ldots,a_n,a_1,\ldots,a_n,\ldots],\]
i.e., the reduced quadratic irrationals $\omega\in[0,1]$ with conjugate root $\bar{\omega}<-1$; and the periodic points of $\tilde{G}$ are of the form
\[([\overline{a_1,a_2,\ldots,a_n}],[\overline{a_n,a_{n-1},\ldots,a_1}]),\]
where the continued fraction expansion of the second argument is the reverse of that of the first. Alternatively, the periodic points of $\tilde{G}$ are of the form $(\omega,-\bar{\omega}^{-1})$, where $\omega$ is a reduced quadratic irrational. Notice that $\omega\leftrightarrow(\omega,-\bar{\omega}^{-1})$ gives a natural one-to-one correspondence between the periodic points of $G$ and $\tilde{G}$.

Let $Q_G$ denote the set of nonzero periodic points of $G$. To each $\omega\in Q_G$ with minimal even periodic expansion $\omega=[\overline{a_1,\ldots,a_{2n}}]$, we associate the value
\begin{equation}\label{length1}
\ell(\omega):=-2\sum_{j=1}^{2n}\log(G^j(\omega))
\end{equation}
which is the length of a corresponding geodesic in the modular surface (see Section \ref{ssec:Series}). For future reference, we analogously define, for a given tuple $a=(a_1,\ldots,a_n)\in\N^n$ of any length,
\[\ell(a):=-2\sum_{j=1}^{n}\log(G^j[\overline{a_1,a_2,\ldots,a_n}]).\]
 We then let
\[Q_G(T):=\{\omega\in Q_G:\ell(\omega)\leq T\}.\qquad(T>0)\]
The result of Pollicott \cite[Theorem 3]{P} states that for all $f\in C([0,1])$, we have
\begin{equation}\label{Gres1}
\lim_{T\rar\infty}\frac{1}{|Q_G(T)|}\sum_{\omega\in Q_G(T)}f(\omega)=\int_{[0,1]} f\,d\nu;
\end{equation}
and it then follows from Kelmer's result \cite[Lemma 17]{Ke} that for all $f\in C([0,1]^2)$,
\begin{equation}\label{Gres2}
\lim_{T\rar\infty}\frac{1}{|Q_G(T)|}\sum_{\omega\in Q_G(T)}f(\omega,-\bar{\omega}^{-1})=\int_{[0,1]^2}f\,d\tilde{\nu}.
\end{equation}

\subsection{The Farey map}\label{ssec:Fmap}
The main goal of this paper is to formulate and prove results analogous to \eqref{Gres1} and \eqref{Gres2} for the periodic points of the Farey map and its natural extension. The Farey map $F:[0,1]\rar[0,1]$ is defined by
\[F(x):=\begin{cases}
\displaystyle{\frac{x}{1-x}}&\text{if $0\leq x\leq\frac{1}{2}$}\\[6pt]
\displaystyle{\frac{1-x}{x}}&\text{if $\frac{1}{2}<x\leq1$}.
\end{cases}\]
The invariant measure for $F$ that is absolutely continuous with respect to the Lebesgue measure is the infinite measure $\mu$ given by
\[d\mu:=\frac{dx}{x}.\]
The natural extension $\tilde{F}:[0,1]^2\rar[0,1]^2$ of $F$ is defined by
\[\tilde{F}(x,y):=\begin{cases}
\displaystyle{\left(\frac{x}{1-x},\frac{y}{1+y}\right)}&\text{if $0\leq x\leq\frac{1}{2}$}\\[8pt]
\displaystyle{\left(\frac{1-x}{x},\frac{1}{1+y}\right)}&\text{if $\frac{1}{2}<x\leq1$},
\end{cases}\]
and has infinite invariant measure $\tilde{\mu}$ given by
\[d\tilde{\mu}:=\frac{dx\,dy}{(x+y-xy)^2}.\]
For future reference, we note that $\mu$ is the natural projection of $\tilde{\mu}$ onto the first coordinate, i.e.,
\begin{equation}\label{muequality}
\int_{[0,1]^2}f(x)\,d\tilde{\mu}(x,y)=\int_{[0,1]}f\,d\mu.\qquad(f\in L^1(\mu))
\end{equation}
The maps $F$ and $\tilde{F}$ act on continued fractions according to
\begin{align*}
F([a_1,a_2,\ldots])&=\begin{cases}
[a_1-1,a_2,\ldots]&\text{if $a_1\geq2$}\\
[a_2,a_3,\ldots]&\text{if $a_1=1$};
\end{cases}\\
\tilde{F}([a_1,a_2,\ldots],[b_1,b_2,\ldots])&=\begin{cases}
([a_1-1,a_2,\ldots],[b_1+1,b_2,\ldots])&\text{if $a_1\geq2$}\\
([a_2,a_3,\ldots],[1,b_1,b_2,\ldots])&\text{if $a_1=1$}.
\end{cases}
\end{align*}
So, like $G$ and $\tilde{G}$, the Farey map and its extension act as shifts on the continued fraction expansions of its arguments, though in a slower manner, shifting a $1$ in a digit instead of a whole digit at a time. In fact, $F$ is a slowdown of $G$ as demonstrated by the equality
\[F^{\lfloor x^{-1}\rfloor}(x)=G(x).\qquad(x\neq0)\]
Also, by \cite[Theorem 1]{BY}, $G$ is isomorphic to the induced transformation $F_{A}:A\rar A$ of the Farey map on $A:=[1/2,1]=\{[1,a_1,a_3,\ldots]:a_j\in\N\}$ defined by
\[F_A(x)=F^{n_A(x)}(x),\quad\text{where}\quad n_A(x)=\min\{n\in\N:F^n(x)\in A\}.\]
This can be easily seen from the equality
\[F_A([1,a_1,a_2,\ldots])=[1,a_2,a_3,\ldots].\]
Similarly, $\tilde{G}$ can be seen as isomorphic to the induced transformation $\tilde{F}_{\tilde{A}}:\tilde{A}\rar\tilde{A}$ of $\tilde{F}$ on $\tilde{A}=(0,1]\times(1/2,1]$ defined by
\begin{equation}\label{inducedmap1}
\tilde{F}_{\tilde{A}}(x)=\tilde{F}^{n_{\tilde{A}}(x)}(x),\quad\text{where}\quad n_{\tilde{A}}(x)=\min\{n\in\N:\tilde{F}^n(x)\in\tilde{A}\},
\end{equation}
from the equality
\begin{equation}\label{inducedmap2}
\tilde{F}_{\tilde{A}}([a_1,a_2,\ldots],[1,b_1,b_2,\ldots])=([a_2,a_3,\ldots],[1,a_1,b_1,b_2,\ldots]).
\end{equation}
In Section \ref{Sec3}, we see how this relationship between $\tilde{G}$ and $\tilde{F}$ suggests to us how to enlarge the cross section of Series in a way analogous to how one might enlarge the domain of $\tilde{G}$ to obtain a map acting like $\tilde{F}$.

Early studies of the Farey map include \cite{Feig,FPT} in the context of thermodynamics, and \cite{Ito}, where the natural extension $\tilde{F}$ was also introduced, in examining mediant continued fraction convergents. See also \cite{BY} for more details on the above properties of $F$ and $\tilde{F}$ and their relationships to $G$ and $\tilde{G}$, in addition to continued fraction applications.

We have the following characterizations of the periodic points of $F$ and $\tilde{F}$ which provide connections to those of $G$ and $\tilde{G}$.

\begin{proposition}\label{P1}\ 
\begin{itemize}
\item[(i)] A number $x\in(0,1]$ is a periodic point of $F$ if and only if
\[x=[a_1-k,\overline{a_2,\ldots,a_n,a_1}],\]
for some $a_j\in\N$ and $k\in\{0,\ldots,a_1-1\}$. In other words, the nonzero periodic points of $F$ are of the form $F^k(\omega)$, where $\omega\in Q_G$.
\item[(ii)] A point in $[0,1]^2\bsl\{(0,0)\}$ is a periodic point of $\tilde{F}$ if and only if it is of the form
\begin{equation*}
\tilde{F}^k([\overline{a_1,a_2,\ldots,a_n}],[1,\overline{a_n,a_{n-1}\ldots,a_1}])
=([a_1-k,\overline{a_2,a_3,\ldots,a_n,a_1}],1+k,\overline{a_n,a_{n-1}\ldots,a_1}])
\end{equation*}
for some $a_j\in\N$ and $k\in\{0,\ldots,a_1-1\}$. Equivalently, the nonzero periodic points of $\tilde{F}$ are of the form $\tilde{F}^k(x)$, where $x$ is a periodic point of the induced map $\tilde{F}_{\tilde{A}}$.
\end{itemize}
\end{proposition}

The proof of this proposition is elementary, and so is omitted. In short, these characterizations follow in a straightforward manner from the fact that the intersection of $A$ ($\tilde{A}$ respectively) with any periodic orbit of $F$ ($\tilde{F}$ respectively) must be a periodic orbit of $F_A$ ($\tilde{F}_{\tilde{A}}$ respectively).

Notice that, analogous to the periodic points of $G$ and $\tilde{G}$, there is a natural correspondence
\[[a_1-k,\overline{a_2,\ldots,a_n,a_1}]\leftrightarrow([a_1-k,\overline{a_2,\ldots,a_n,a_1}],[1+k,\overline{a_n,\ldots,a_1}])\]
between the periodic points of $F$ and $\tilde{F}$. We let $Q_F$ be the set of all nonzero periodic points of $F$, and for a given $\omega\in Q_F$, we let $\tilde{\omega}\in[0,1]$ be such that $(\omega,\tilde{\omega})$ is the periodic point of $\tilde{F}$ corresponding to $\omega$.

We extend the definition of the length function $\ell$ on $Q_G$ to $Q_F$ by letting
\[\ell(F^k(\omega)):=\ell(\omega)\]
for all $\omega\in Q_G$ and $k\in\N$. We shall see that this definition follows naturally from the correspondence between the primitive closed geodesics in the modular surface and the periodic points of the Farey map which we develop in Section \ref{Sec3}. Also, define the set
\[Q_F(T)=\{\omega\in Q_F:\ell(\omega)\leq T\}.\]

We can now formulate our main theorem. It is best expressed in terms of proving the equidistribution in $[0,1]$ of the weighted points in $Q_F(T)$ as $T\rar\infty$. For $T>0$, we define the measure $m_T$ on $[0,1]$ by the equality
\[\int_{[0,1]}f\,dm_T:=\frac{\sum_{\omega\in Q_F(T)}\omega f(\omega)}{\sum_{\omega\in Q_F(T)}\omega}.\qquad (f\in C([0,1]))\]
In other words, $m_T$ is the sum of the Dirac delta measures over the points $\omega\in Q_F(T)$ with weight $\omega$, normalized to be a probability measure. Our main result is the following:

\begin{theorem}\label{T1}
For $f\in C([0,1])$, we have
\[\lim_{T\rar\infty}\int_{[0,1]}f\,dm_T=\int_0^1f(x)\,dx.\]
In other words, the weighted set of periodic points of $F$ given by the measure $m_T$ equidistributes with respect to the Lebesgue measure on $[0,1]$.
\end{theorem}

Alternatively, one may view this theorem as saying that the unweighted periodic points of $F$ equidistribute according to the invariant measure $\mu$. However, one must restrict the functions over which equidistribution can be tested to those of the form $x\mapsto xf(x)$, with $f\in C([0,1])$. Additionally, one must maintain the normalizing factor $\sum_{\omega\in Q_F(T)}\omega$; because $\mu$ has infinite measure, normalizing by $\#Q_F(T)$ would yield $0$ in the limit. Indeed, we see below that the growth rate of $\sum_{\omega\in Q_F(T)}\omega$ is commensurate with $e^T$. However, the growth of $\#Q_F(T)$ is given by
\begin{equation}\label{QFgrowth}
\#Q_F(T)=\frac{1}{4\zeta(2)}Te^T+\frac{1}{2\zeta(2)}\left(\gamma-\frac{3}{2}-\frac{\zeta'(2)}{\zeta(2)}\right)e^T+O(T^4e^{3T/4}),\qquad(T\rar\infty)
\end{equation}
where $\zeta$ is the Riemann zeta function and $\gamma$ is Euler's constant. This follows from the fact that
\[\#Q_F(T)=\sum_{\omega\in Q_G(T)}\lfloor\omega^{-1}\rfloor\]
(for each $\omega=[\overline{a_1,\ldots,a_n}]\in Q_G(T)$ there exist $a_1=\lfloor\omega^{-1}\rfloor$ corresponding elements of $Q_F(T)$); and in analyzing the growth of the number of products of the matrices $\left(\begin{smallmatrix}1&1\\0&1\end{smallmatrix}\right)$ and $\left(\begin{smallmatrix}1&0\\1&1\end{smallmatrix}\right)$ with bounded trace, Kallies et al.\ \cite{KOPS} provided an asymptotic expression for the sum on the right, giving the main term in \eqref{QFgrowth}. The second term was extracted by Boca \cite{Bo}, who then obtained the error term $O(e^{(7/8+\epsilon)T})$. Ustinov \cite{Ust} later analyzed the error term more carefully, and obtained that shown in \eqref{QFgrowth}.

\begin{remark}
Results analogous to Theorem \ref{T1} formulated for certain maps $t:[0,1]\rar[0,1]$ that are uniformly expanding as in, for example, \cite{Mo}, establish the equidistribution of the unweighted periodic points $\{\omega\in[0,1]:t^n(\omega)=\omega,\log((t^n)'(\omega))\leq T\}$ with respect to the absolutely continuous probability measure of $t$.

Similarly, the equidistribution of the periodic points of the Farey map following from \cite{PU} is formulated for the unweighted set $Q_F'(T)=\{\omega\in[0,1]:F^n(\omega)=\omega,\log((F^n)'(\omega))\leq T\}$, though any subset of $[0,1]$ over which equidistribution is tested must be of finite $\mu$-measure. Specifically, \cite[Theorem 18.1]{PU} says that if $B\subseteq[0,1]$ such that $\mu(B)<\infty$ and $m(\partial B)=0$, then
\[\#(Q_F'(T)\cap B)\sim\frac{6\mu(B)}{\pi^2}e^T.\qquad(T\rar\infty)\]
We obtain the very similar asymptotic formula \eqref{asympt2} below. The difference in our results stems primarily the fact that while $\ell(\omega)=\log((F^n)'(\omega))$ if $n\in\N$ is the least element such that $F^n(\omega)=\omega$ and $F^k(\omega)$ is a periodic continued fraction of even period for some $k$, we have $\ell(\omega)=2\log((F^n)'(\omega))$ if $F^k(\omega)$ is a periodic continued fraction of odd period for some $k$. Thus the periodic points of $F$ corresponding to odd continued fraction periods are given less weight in our situation (see Section \ref{ssec:equivalent}).
\end{remark}

\begin{remark}
If we define $\tilde{Q}_F(T)$ to be the set of $\omega\in Q_F$ such that $F^n(\omega)=\omega$ for some $n\leq T$, then using the fact that $F$ and the tent map $t:[0,1]\rar[0,1]$, $t(x)=\min\{2x,2-2x\}$, are conjugate via the Minkowski question mark function $?$ \cite{Mi,KS}, it is straightforward to show that as $T\rar\infty$, $\tilde{Q}_F(T)$ equidistributes with respect to the measure having distribution function $?$. Indeed, $?$ maps $\tilde{Q}(T)$ to the periodic points of $t$ of period at most $T$, which can be easily shown to equidistribute with respect to the Lebesgue measure. In fact, it is elementary to show that the periodic points of $t$ of period exactly $T$ ($T\in\N$) equidistribute, implying the corresponding equidistribution of the points $\tilde{Q}_F(T)\bsl\bigcup_{T'<T}\tilde{Q}_F(T')$.
\end{remark}

Using the correspondence between the periodic points of $F$ and $\tilde{F}$, we also obtain an analogous equidistribution result for the periodic points of $\tilde{F}$ as a corollary of Theorem \ref{T1}. Define the function $h:[0,1]^2\rar\R$ by $h(x,y)=(x+y-xy)^2$ and the measure $\tilde{m}_T$ on $[0,1]^2$ by
\[\int_{[0,1]^2}f\,d\tilde{m}_T:=\frac{\sum_{\omega\in Q_F(T)}h(\omega,\tilde{\omega})f(\omega,\tilde{\omega})}{\sum_{\omega\in Q_F(T)}h(\omega,\tilde{\omega})}.\qquad(f\in C([0,1]^2))\]
We then have the following:

\begin{corollary}\label{C1}
For all $f\in C([0,1]^2)$,
\[\lim_{T\rar\infty}\int_{[0,1]^2}f\,d\tilde{m}_T=\int_0^1\int_0^1f(x,y)\,dx\,dy,\]
that is, the weighted sequence of periodic points of $\tilde{F}$ given by $\tilde{m}_T$ equidistributes with respect to the Lebesgue measure on $[0,1]^2$.
\end{corollary}

\begin{proof}
We begin by following the reasoning of \cite[Lemma 17]{Ke} and showing the asymptotic formula
\begin{equation}\label{asympt1}
\lim_{T\rar\infty}\frac{\sum_{\omega\in Q_F(T)}f(\omega,\tilde{\omega})}{\sum_{\omega\in Q_F(T)}\omega}=\int_{[0,1]^2}f\,d\tilde{\mu}
\end{equation}
for appropriate approximating functions $f$. We first verify \eqref{asympt1} when $f$ is an indicator function $1_{I_b\times I_{b'}}$, where $b=(b_1,\ldots,b_n)$ and $b'=(b'_1,\ldots,b'_{n'})$ are any tuples and $I_b=\llbr b\rrbr$ and $I_{b'}=\llbr b'\rrbr$ are the sets defined in Section \ref{ssec:cfracs}. First note that if $B=b_1'+\cdots+b'_{n'}$, then $1_{I_b\times I_{b'}}\circ\tilde{F}^B=1_{I_{b''}\times[0,1]}$, where
\[b''=(1,b'_n,b'_{n-1},\ldots,b'_2,b'_1+b_1-1,b_2,b_3,\ldots,b_n).\]
Also, since $\tilde{F}$ forms a bijection of $\{(\omega,\tilde{\omega}):\omega\in Q_F(T)\}$, we have
\[\sum_{\omega\in Q_F(T)}1_{I_b\times I_{b'}}(\omega,\tilde{\omega})=\sum_{\omega\in Q_F(T)}(1_{I_b\times I_{b'}}\circ\tilde{F}^j)(\omega,\tilde{\omega})\]
for any $j\in\Z$. Therefore, using Theorem \ref{T1}, the equality \eqref{muequality}, and the $\tilde{F}$-invariance of $\tilde{\mu}$, we have
\begin{align*}
&\lim_{T\rar\infty}\frac{\sum_{\omega\in Q_F(T)}1_{I_b\times I_{b'}}(\omega,\tilde{\omega})}{\sum_{\omega\in Q_F(T)}\omega}
=\lim_{T\rar\infty}\frac{\sum_{\omega\in Q_F(T)}(1_{I_b\times I_{b'}}\circ\tilde{F}^B)(\omega,\tilde{\omega})}{\sum_{\omega\in Q_F(T)}\omega}\\
&\qquad\qquad\qquad=\lim_{T\rar\infty}\frac{\sum_{\omega\in Q_F(T)}1_{I_{b''}\times[0,1]}(\omega,\tilde{\omega})}{\sum_{\omega\in Q_F(T)}\omega}
=\lim_{T\rar\infty}\frac{\sum_{\omega\in Q_F(T)}1_{I_{b''}}(\omega)}{\sum_{\omega\in Q_F(T)}\omega}
=\int_{[0,1]}1_{I_{b''}}\,d\mu\\
&\qquad\qquad\qquad=\int_{[0,1]^2}1_{I_{b''}\times[0,1]}\,d\tilde{\mu}
=\int_{[0,1]^2}1_{I_{b''}\times[0,1]}\circ\tilde{F}^B\,d\tilde{\mu}
=\int_{[0,1]^2}1_{I_b\times I_{b'}}\,d\tilde{\mu}.
\end{align*}
We thus have \eqref{asympt1} for $f$ of the form $1_{I_b\times I_{b'}}$.

Next, \eqref{asympt1} can be easily verified when $f(x,y)=x$. Also, we see in Section \ref{ssec:Series} that the set $\{(\omega,\tilde{\omega}):\omega\in Q_F(T)\}$ is symmetric about the line $x=y$. As a result, \eqref{asympt1} holds also when $f(x,y)=y$.

We now have a sufficient set of approximating functions, so let $f\in C([0,1]^2)$. By splitting $f$ into its positive and negative parts, we may assume without loss of generality that $f\geq0$. For a given $\epsilon>0$, there exists a finite linear combination, which we denote by $f_\epsilon$, of indicator functions of the form $1_{I_b\times I_{b'}}$ such that $f_\epsilon\leq h\cdot f$ and
\[\int_{[0,1]^2}(h\cdot f - f_\epsilon)\,d\tilde{\mu}<\epsilon.\]
This is possible since the sets of the form $I_b\times I_{b'}$ generate the Borel $\sigma$-algebra of $[0,1]^2$ and
\[\int_{[0,1]^2}h\cdot f\,d\tilde{\mu}=\int_0^1\int_0^1 f(x,y)\,dx\,dy<\infty.\]
We then have
\begin{align*}
\liminf_{T\rar\infty}\frac{\sum_{\omega\in Q_F(T)}h(\omega,\tilde{\omega})f(\omega,\tilde{\omega})}{\sum_{\omega\in Q_F(T)}\omega}&\geq\lim_{T\rar\infty}\frac{\sum_{\omega\in Q_F(T)}f_\epsilon(\omega,\tilde{\omega})}{\sum_{\omega\in Q_F(T)}\omega}=\int_{[0,1]^2}f_\epsilon\,d\tilde{\mu}\\
&\geq\int_{[0,1]^2}h\cdot f\,d\tilde{\mu}-\epsilon.
\end{align*}
Letting $\epsilon\rar0$ yields
\begin{equation}\label{ineq1}
\liminf_{T\rar\infty}\frac{\sum_{\omega\in Q_F(T)}h(\omega,\tilde{\omega})f(\omega,\tilde{\omega})}{\sum_{\omega\in Q_F(T)}\omega}\geq\int_{[0,1]^2}h\cdot f\,d\tilde{\mu}.
\end{equation}
Now notice that $h(x,y)\cdot f(x,y)\leq H(x,y):=\|f\|_\infty(x+y)$ for $(x,y)\in[0,1]^2$. Repeating the above process used to produce the inequality \eqref{ineq1}, while replacing the function $h\cdot f$ with $H-h\cdot f$, we find that
\[\liminf_{T\rar\infty}\frac{\sum_{\omega\in Q_F(T)}(H-h\cdot f)(\omega,\tilde{\omega})}{\sum_{\omega\in Q_F(T)}\omega}\geq\int_{[0,1]^2}(H-h\cdot f)\,d\tilde{\mu}.\]
Since \eqref{asympt1} is satisfied when $f$ is replaced by $H$, we can cancel the $H$ in the above inequality to get
\[\limsup_{T\rar\infty}\frac{\sum_{\omega\in Q_F(T)}h(\omega,\tilde{\omega})f(\omega,\tilde{\omega})}{\sum_{\omega\in Q_F(T)}\omega}\leq\int_0^1\int_0^1f(x,y)\,dx\,dy,\]
and thus
\[\lim_{T\rar\infty}\frac{\sum_{\omega\in Q_F(T)}h(\omega,\tilde{\omega})f(\omega,\tilde{\omega})}{\sum_{\omega\in Q_F(T)}\omega}=\int_0^1\int_0^1f(x,y)\,dx\,dy.\]
Dividing this equality by the same equality, with $f$ replaced by the constant function $1$, yields the result.
\end{proof}

In an analogous manner to Theorem \ref{T1}, one can view Corollary \ref{C1} as the equidistribution of the unweighted periodic points of $\tilde{F}$ according to the measure $\tilde{\mu}$. As before, one must restrict the functions over which to test the equidistribution and maintain the normalization $\sum_{\omega\in Q_F(T)}h(\omega,\tilde{\omega})$. The appendix shows how the analysis of the number of products of $\left(\begin{smallmatrix}1&1\\0&1\end{smallmatrix}\right)$ and $\left(\begin{smallmatrix}1&0\\1&1\end{smallmatrix}\right)$ with bounded trace alluded to above can be used to obtain unweighted variations of Theorem \ref{T1} and Corollary \ref{C1} which have error terms.

\section{Geodesics in the modular surface and the Farey map}\label{Sec3}

\subsection{The modular surface and the geodesic flow}\label{ssec:modular}
Let $\Hp=\{x+iy:x,y\in\R,y>0\}$ denote the upper half of the complex plane equipped with the hyperbolic metric $ds^2=(dx^2+dy^2)/y^2$. The geodesics in $\Hp$ with this metric are vertical lines and the semicircles centered on the real line. The group $\PSL_2(\R)$ acts isometrically on $\Hp$ by linear fractional transformation. The modular surface is the quotient space $\mathcal{M}:=\PSL_2(\Z)\bsl\Hp$, whose geodesics are naturally projected from those of $\Hp$.

Let $T_1\Hp$ and $T_1\mathcal{M}$ be the unit tangent bundles of the upper half plane and modular surface, respectively. Then let $g_t:T_1\Hp\rar T_1\Hp$ denote the geodesic flow on $T_1\Hp$ so that for $(z,v)\in T_1\Hp$, $g_t(z,v)$ is the tangent vector obtained by starting at the base point $z$, and moving a distance $t$ along the geodesic tangent to the vector $(z,v)$. Let $\{g_t:t\in\R\}$ also be defined on $T_1\mathcal{M}$ by natural projection. Next, define the coordinates $(x,y,\theta)$ on $T_1\Hp$ (and, locally, on $T_1\mathcal{M}$ by projection) corresponding to the vector with base point $x+iy$ and at an angle $\theta$ counterclockwise from the vertical vector. Then the $g_\cdot$-invariant Liouville measure $\lambda$ given by $d\lambda:=dx\,dy\,d\theta/y^2$ is obtained by importing the Haar measure on $\PSL_2(\R)$ to $T_1\Hp$ via the natural identification
\begin{equation}\label{PSLid}
\gamma\mapsto\gamma(i,v_0)=(\gamma(i),\gamma'(i)v_0):\PSL_2(\R)\xrightarrow{\sim}T_1\Hp,
\end{equation}
where $v_0$ is the upward-pointing vector at $i$. The measure $\lambda$ naturally descends to $T_1\mathcal{M}$. Note also that under the above identification, the geodesic flow $g_t$ corresponds to right multiplication by
\[\left(\begin{array}{cc}e^{t/2}&0\\0&e^{-t/2}\end{array}\right).\]

We can use the alternative coordinates  $(\alpha,\beta,t)\in\R^3$ on $T_1\Hp$, which correspond to the point $(z,v)\in T_1\Hp$ such that $\alpha=\alpha(z,v):=\lim_{t\rar-\infty}g_t(z,v)$ is the endpoint of the geodesic $g_{(z,v)}:=\{g_s(z,v):s\in\R\}$ approached by $(z,v)$ under the geodesic flow in the backward direction, $\beta=\beta(z,v):=\lim_{t\rar\infty}g_t(z,v)$ is the endpoint of $g_{(z,v)}$ approached under the flow in the forward direction, and $t=t(z,v)$ is such that $g_{-t(z,v)}(z,v)$ is the apex of $g_{(z,v)}$. With respect to these coordinates, the Liouville measure is
\[d\lambda=\frac{d\alpha\,d\beta\,dt}{(\beta-\alpha)^2}.\]

\subsection{The cross section of Series and $\tilde{G}$}\label{ssec:Series}
The cross section of the geodesic flow considered by Series \cite{S} is
\[X=\{(z,v)\in T_1\mathcal{M}:0<|\alpha(z,v)|\leq1,|\beta(z,v)|\geq1,z\in i\R\}.\]
If $\beta(z,v)\neq\pm1$, then the geodesic flow returns $(z,v)$ to $X$. So the first return map $P$ for $X$ defined by
\[P(z,v):=g_{r(z,v)}(z,v),\quad\text{where}\quad r(z,v):=\min\{s>0:g_s(z,v)\in X\},\]
is well defined on $X^*:=\{(z,v)\in X:|\beta(z,v)|>1\}$. (Note that $X^*$ is of full measure in $X$ with respect to the measure $d\alpha\,d\beta/(\beta-\alpha)^2$ induced on $X$ by $\lambda$.) Using the correspondence between the continued fraction expansions of $\alpha(z,v)$ and $\beta(z,v)$ and the cutting sequence of the geodesic $g_{(z,v)}$, Series proved that $\tilde{G}$ is a factor of $P$. Specifically, if we parameterize $X$ by the coordinates $(U,V,\epsilon)\in(0,1]^2\times\{\pm1\}$, where $U=|\beta|^{-1}$, $V=|\alpha|$, and $\epsilon=\sign(\beta)$, and abuse notation by identifying $X$ with $(0,1]^2\times\{\pm1\}$, then
\[P(U,V,\epsilon)=(\tilde{G}(U,V),-\epsilon).\]
(Here we must also assume that $U^{-1}\notin\Z$.) So if we define $\pi_X:X\rar(0,1]^2$ by $\pi_X(U,V,\epsilon)=(U,V)$, then we have the commutative diagram
\[\xymatrix{
X^* \ar[d]^{\pi_X} \ar[r]^P & X \ar[d]^{\pi_X}\\
[0,1]^2 \ar[r]^{\tilde{G}} & [0,1]^2}\]
which expresses $\tilde{G}$ as a factor of $P$.

The cross section $X$ also allows us to see that there is a one-to-one correspondence between the closed geodesics in $T_1\mathcal{M}$ and periodic orbits of the return map $P$. Indeed, a closed geodesic is simply the $g_\cdot$-orbit in $T_1\mathcal{M}$ of one of points in a periodic orbit of $P$. A given periodic orbit of $P$ is of the form
\begin{equation}\label{orbit1}
\{(\tilde{G}^j([\overline{a_1,a_2,\ldots,a_{2n}}],[\overline{a_{2n},a_{2n-1},\ldots,a_1}]),\pm(-1)^j):j=1,\ldots,2n\},
\end{equation}
where $2n$ is the minimal even period length of $[\overline{a_1,a_2,\ldots,a_{2n}}]$. It follows from \cite[Section 3.2]{S} that the length of the geodesic corresponding to the orbit \eqref{orbit1} is
\[-2\sum_{j=1}^{2n}\log(G^j([\overline{a_1,a_2,\ldots,a_{2n}}])),\]
which inspired the definition \eqref{length1} in \cite{P}. Here, we wish to note that the closed geodesics corresponding to the orbits
\begin{align*}
&\{(\tilde{G}^j([\overline{a_1,a_2,\ldots,a_{2n}}],[\overline{a_{2n},a_{2n-1},\ldots,a_1}]),\pm(-1)^j):j=1,\ldots,2n\}\text{, and}\\
&\{(\tilde{G}^j([\overline{a_{2n},a_{2n-1},\ldots,a_1}],[\overline{a_1,a_2,\ldots,a_{2n}}]),\pm(-1)^j):j=1,\ldots,2n\}
\end{align*}
are permuted by the symmetry $(z,v)\mapsto(z,-v)$ on $T_1\mathcal{M}$. Indeed, the first orbit corresponds to the geodesic tangent to the element $(z,v)\in T_1\mathcal{M}$ with
\[\beta(z,v)=\pm[\overline{a_1;a_2,\ldots,a_{2n}}]\quad\text{and}\quad\alpha(z,v)=\mp[\overline{a_{2n},a_{2n-1},\ldots,a_1}],\]
whereas the second orbit corresponds to the geodesic tangent to the element $(z',v')\in T_1\mathcal{M}$ such that $z'\in i\R$,
\[\beta(z',v')=\pm[\overline{a_{2n};a_{2n-1},\ldots,a_1}],\quad\text{and}\quad\alpha(z',v')=\mp[\overline{a_1,a_2,\ldots,a_{2n}}].\]
Acting by the matrix $\left(\begin{smallmatrix}0&-1\\1&0\end{smallmatrix}\right)\in\PSL_2(\Z)$, we see that this geodesic is the same as that tangent to $(z'',v'')\in T_1\Hp$, where $z''=-(z')^{-1}$,
\[\beta(z'',v'')=\mp[\overline{a_{2n},a_{2n-1},\ldots,a_1}],\quad\text{and}\quad\alpha(z',v')=\pm[\overline{a_1;a_2,\ldots,a_{2n}}].\]
In other words, we have $(z'',v'')=(z,-v)$. Hence the two geodesics are the same length, i.e.,
\begin{equation}\label{lengtheq1}
\ell([\overline{a_1,a_2,\ldots,a_{2n}}])=\ell([\overline{a_{2n},a_{2n-1},\ldots,a_1}]).
\end{equation}

\subsection{A new cross section for $\tilde{F}$}\label{ssec:Fsection}
We now seek to find a cross section analogous to $X$ whose return map under the geodesic flow is a double cover of $\tilde{F}$. The fact that $\tilde{G}$ is isomorphic to the induced map \eqref{inducedmap1} hints that we should seek to expand the range of the parameter $V$ with respect to the coordinates $(U,V,\epsilon)$, or the range of $\alpha$ with respect to the coordinates $(\alpha,\beta)$. In fact, we define our new cross section $\bar{X}$ by
\[\bar{X}:=\{(z,v)\in T_1\mathcal{M}:\alpha(z,v)\neq0,|\beta(z,v)|\geq1,z\in i\R\};\]
so we have just removed the restriction $|\alpha|\leq1$ from the endpoint $\alpha$ of the geodesic determined by $(z,v)$. One can also see that $\bar{X}$ is simply all the nonvertical tangent vectors with base point on the positive imaginary axis. We can parameterize $\bar{X}$ by the coordinates $(U,W,\epsilon)\in(0,1]\times(0,1)\times\{\pm1\}$, where $W=(1+|\alpha|)^{-1}$, and as before, $\bar{U}=|\beta|^{-1}$ and $\epsilon=\sign(\beta)$. (We again abuse notation by identifying $\bar{X}$ with $(0,1]\times(0,1)\times\{\pm1\}$.) Our definition of the coordinate $W$ follows from the equality \eqref{inducedmap2}, which motivates us to change the second coordinates of our points in $X$ according to the map
\[[b_1,b_2,\ldots]\mapsto[1,b_1,b_2,\ldots].\]

Now let $\bar{P}:\bar{X}^*\rar\bar{X}$ be the first return map
\[\bar{P}(z,v):=g_{\bar{r}(z,v)}(z,v),\quad\text{where}\quad\bar{r}(z,v):=\min\{s>0:g_s(z,v)\in\bar{X}\},\]
where $\bar{X}^*=(0,1)^2\times\{\pm1\}$ is the set of points in $\bar{X}$ on which $\bar{P}$ is defined. Our main goal of this section is to prove the following:

\begin{theorem}\label{T2}
The natural extension of the Farey map $\tilde{F}$ is a factor of the return map $\bar{P}$. Specifically, we have
\begin{equation}\label{Pbar}
\bar{P}(U,W,\epsilon)=\begin{cases}
(\tilde{F}(U,W),\epsilon)&\text{if $0<U\leq\frac{1}{2}$}\\
(\tilde{F}(U,W),-\epsilon)&\text{if $\frac{1}{2}<U<1$}
\end{cases}
\end{equation}
so that if $\bar{\pi}_{\bar{X}}:X\rar[0,1]^2$ is defined by $\bar{\pi}_{\bar{X}}(U,W,\epsilon)=(U,W)$, we have the following commutative diagram:
\[\xymatrix{
\bar{X}^* \ar[d]^{\bar{\pi}_{\bar{X}}} \ar[r]^{\bar{P}} & \bar{X} \ar[d]^{\bar{\pi}_{\bar{X}}}\\
[0,1]^2 \ar[r]^{\tilde{F}} & [0,1]^2}\]
Also, the first return time function $\bar{r}$ is given by
\begin{equation}\label{Pbarrtf}
\bar{r}(U,W,\epsilon)=-\frac{1}{2}\log((1-U)(1-W)).
\end{equation}
\end{theorem}

\begin{proof}
To begin, it is helpful to note that the set of base points of the vectors lifted from $\bar{X}$ to $T_1\Hp$ lie on the $\PSL_2(\Z)$ translates of the positive imaginary axis, which make up the Farey tesselation. Each of these translates is either a vertical geodesic with integer real part, or a semicircle connecting rational numbers $p/q<p'/q'$ satisfying $p'q-pq'=1$. This is easy to see by evaluating $\lim_{t\rar0,\infty}\gamma(it)$ for $\gamma\in\PSL_2(\Z)$. The important thing to note is that all of the nonvertical translates lie below the semicircles connecting adjacent integers. Part of the Farey tessellation is depicted in Figures \ref{fig1} and \ref{fig2}.

Now let $(z,v)\in\bar{X}^*$, which we identify with its coordinates $(U,W,\epsilon)\in(0,1)^2\times\{\pm1\}$, and assume for simplicity that $\epsilon=1$. (The argument for $\epsilon=-1$ mirrors the following.) Let $U=[a_1,a_2,\ldots]$ and $W=[b_1,b_2,\ldots]$, either of which may terminate. We first consider the case when $U\in(0,1/2]$, and hence $a_1\geq2$. Then the geodesic $g_{(z,v)}$ has endpoints $\beta=\beta(z,v)=U^{-1}=[a_1;a_2,\ldots]$ and $\alpha=\alpha(z,v)=-(V^{-1}-1)=-[b_1-1;b_2,\ldots]$. Since $\beta>1$, the first point in $\bar{X}$ that $(z,v)$ encounters under the geodesic flow is the point $(z',v')$, where $\Rl(z')=1$. The matrix $\left(\begin{smallmatrix}1&-1\\0&1\end{smallmatrix}\right)\in\PSL_2(\Z)$ identifies $(z',v')$ with the point $(z'',v'')$, where $\beta(z'',v'')=\beta-1=[a_1-1;a_2,\ldots]$ and $\alpha(z'',v'')=\alpha-1=-[b_1;b_2,\ldots]$. (See Figure \ref{fig1}.) The $(U,W,\epsilon)$ coordinates of this element are $U''=[a_1-1,a_2,\ldots]$, $W''=[b_1+1,b_2,b_3,\ldots]$, and $\epsilon''=1$. Thus \eqref{Pbar} holds for $U\in(0,1/2]$.

\begin{figure}
\centering
\begin{overpic}[width=0.65\textwidth]{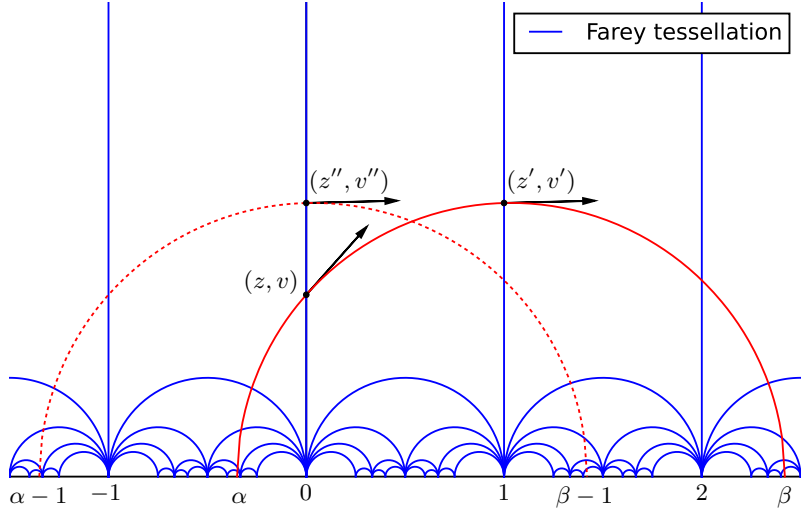}
\put (30,28.5){{\footnotesize $(z,v)$}}
\put (62.6,41){{\footnotesize $(z',v')$}}
\put (38.1,41){{\footnotesize $(z'',v'')$}}
\put (28.5,1.9) {{\footnotesize $\alpha$}}
\put (96.1,1.9) {{\footnotesize $\beta$}}
\put (1,1.9) {{\footnotesize $\alpha-1$}}
\put (68.7,1.9) {{\footnotesize $\beta-1$}}
\put (11,2.3) {{\footnotesize $-1$}}
\put (37,2.3) {{\footnotesize $0$}}
\put (61.5,2.3) {{\footnotesize $1$}}
\put (86,2.3) {{\footnotesize $2$}}
\end{overpic}
\caption{The case $U\leq1/2$}
\label{fig1}
\end{figure}

\begin{figure}
\centering
\begin{overpic}[width=0.65\textwidth]{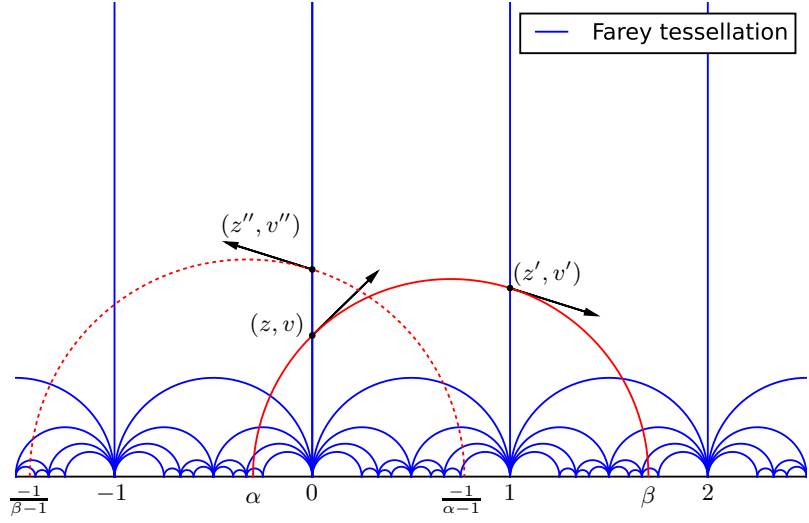}
\put (30,23.2){{\footnotesize $(z,v)$}}
\put (62.6,29.5){{\footnotesize $(z',v')$}}
\put (26.4,35.5){{\footnotesize $(z'',v'')$}}
\put (29.5,1.9) {{\footnotesize $\alpha$}}
\put (78.5,1.9) {{\footnotesize $\beta$}}
\put (53.5,1.6) {{\footnotesize $\frac{-1}{\alpha-1}$}}
\put (0,1.6) {{\footnotesize $\frac{-1}{\beta-1}$}}
\put (11,2.3) {{\footnotesize $-1$}}
\put (37,2.3) {{\footnotesize $0$}}
\put (61.5,2.3) {{\footnotesize $1$}}
\put (86,2.3) {{\footnotesize $2$}}
\end{overpic}
\caption{The case $U>1/2$}
\label{fig2}
\end{figure}

Next, consider the case when $U\in(1/2,1)$, i.e., $a_1=1$. As above, the corresponding geodesic $g_{(z,v)}$ has endpoints $\beta=[a_1;a_2,\ldots]$ and $\alpha=-[b_1-1;b_2,\ldots]$. Similar to the previous case, the first point in $\bar{X}$ that $(z,v)$ passes through under the geodesic flow is the point $(z',v')$, where $\Rl(z')=1$. However, since $\beta(z,v)=[1;a_2,\ldots]<2$, we must use the matrix $\left(\begin{smallmatrix}0&-1\\1&-1\end{smallmatrix}\right)\in\PSL_2(\Z)$ to identify $(z',v')$ with the point $(z'',v'')$ satisfying
\begin{align*}
\beta(z'',v'')=\frac{-1}{\beta-1}=\frac{-1}{[1;a_2,a_3,\ldots]-1}=-[a_2;a_3,a_4,\ldots]\text{ and}\\
\alpha(z'',v'')=\frac{-1}{\alpha-1}=\frac{-1}{-[b_1-1;b_2,b_3,\ldots]-1}=[b_1,b_2,b_3,\ldots].
\end{align*}
(See Figure \ref{fig2}.) The $(U,W,\epsilon)$ coordinates of $(z'',v'')$ are $U''=[a_2,a_3,\ldots]$, $W''=[1,b_1,b_2,\ldots]$, and $\epsilon''=-1$, which shows \eqref{Pbar} for $U\in(1/2,1)$. This proves that $\tilde{F}$ is a factor of $\bar{P}$.

We now outline the calculation of the return time function $\bar{r}$. As above, let $(z,v)\in\bar{X}^*$ with coordinates $(U,V,\epsilon)\in(0,1)^2\times\{\pm1\}$, and again assume that $\epsilon=1$, since the case $\epsilon=-1$ is a mirror image. Also, let $z=iy$ for $y\in\R,y>0$, and $\theta\in(0,\pi)$ be the angle $v$ makes with the upward-pointing vector in the counterclockwise direction. Then under the identification \eqref{PSLid}, $(z,v)$ is identified with
\[\left(\begin{array}{cc}y^{1/2}&0\\0&y^{-1/2}\end{array}\right)\left(\begin{array}{cc}\cos\frac{\theta}{2}&-\sin\frac{\theta}{2}\\\sin\frac{\theta}{2}&\cos\frac{\theta}{2}\end{array}\right).\]
By the previous part of the proof, $\bar{r}(U,V,\epsilon)$ is the constant $t>0$ such that the base point of $g_t(z,v)=(M(i),M'(i)v_0)$, where
\[M=\left(\begin{array}{cc}y^{1/2}&0\\0&y^{-1/2}\end{array}\right)\left(\begin{array}{cc}\cos\frac{\theta}{2}&-\sin\frac{\theta}{2}\\\sin\frac{\theta}{2}&\cos\frac{\theta}{2}\end{array}\right)\left(\begin{array}{cc}e^{t/2}&0\\0&e^{-t/2}\end{array}\right),\]
has real part equal to $1$. By a straightforward calculation, this implies that
\[\bar{r}(U,V,\epsilon)=\frac{1}{2}\log\left(\frac{y\sin(\theta/2)\cos(\theta/2)+\cos^2(\theta/2)}{y\sin(\theta/2)\cos(\theta/2)-\sin^2(\theta/2)}\right).\]
Then, using the fact that $U=y^{-1}\tan\frac{\theta}{2}$ and $W=\frac{1}{1+y\tan(\theta/2)}$, one can easily deduce \eqref{Pbarrtf}.
\end{proof}

\begin{remark}\label{R1}
See \cite[Section 7]{Arn} for a different way of relating the Farey map to a cross section of the geodesic flow in $T^1\mathcal{M}$.
\end{remark}

Hence, we can see how a given periodic orbit
\begin{equation}\label{orbit2}
\left\{\bar{P}^j([\overline{a_1,a_2,\ldots,a_{2n}}],[1,\overline{a_{2n},a_{2n-1},\ldots,a_1}],\pm1):j=1,\ldots,\sum_{k=1}^{2n}a_k\right\}
\end{equation}
of $\bar{P}$ naturally includes the periodic orbit \eqref{orbit1} of $P$, and so corresponds to the same closed geodesic. We therefore extend our definition of the length function $\ell$ to the periodic points of $F$ so that for all $k\in\N$, $\ell(F^k([\overline{a_1,\ldots,a_{2n}}]))$ is the length of the closed geodesic given by the $g_\cdot$-orbit of any point in \eqref{orbit2}, i.e.,
\[\ell(F^k([\overline{a_1,\ldots,a_{2n}}]))=-2\sum_{j=1}^{2n}\log(G^j([\overline{a_1,\ldots,a_{2n}}])).\]
We note here that if $\tau=[\overline{a_1,\ldots,a_{2n}}]$ has minimal even period length $2n$ and $k\in\{0,\ldots,a_1-1\}$, then defining $\omega=F^k(\tau)$, we have
\begin{align*}
\ell(\tilde{\omega})&=\ell([1+k,\overline{a_{2n},a_{2n-1},\ldots,a_1}])
=\ell(F^{a_1-1-k}([\overline{a_1,a_{2n},a_{2n-1},\ldots,a_2}]))\\
&=\ell([\overline{a_1,a_{2n},a_{2n-1},\ldots,a_2}])
=\ell([\overline{a_{2n},a_{2n-1},\ldots,a_1}])
=\ell([\overline{a_1,a_2,\ldots,a_{2n}}])
=\ell(\omega),
\end{align*}
where we used \eqref{lengtheq1} for the penultimate equality. This shows that for any $T>0$, the set $\{(\omega,\tilde{\omega}):\omega\in Q_F(T)\}$ is symmetric about the line $x=y$, a fact we used in the proof of Corollary \ref{C1}.

To conclude this section, we notice that the measure $d\alpha\,d\beta/(\beta-\alpha)^2$ on $\bar{X}$ induced by the Liouville measure is, in the coordinates $(U,W)$ (with $\epsilon$ fixed as $1$ or $-1$),
\[-\frac{d(U^{-1})\,d(-(W^{-1}-1))}{(U^{-1}+(W^{-1}-1))^2}=\frac{dU\,dW}{(U+W-UW)^2}.\]
Thus the Liouville measure naturally induces the invariant measure $\tilde{\mu}$ of $\tilde{F}$ on $[0,1]^2$.

\section{Proof of equidistribution}\label{Sec4}

We now set out to prove Theorem \ref{T1}. The result follows directly from the following asymptotic formula we aim to show, and which holds for all $f\in C([0,1])$:
\begin{equation}\label{asympt2}
\sum_{\omega\in Q_F(T)}\omega f(\omega)\sim\left(\frac{3}{\pi^2}\int_0^1f(x)\,dx\right)e^T.\qquad(T\rar\infty)
\end{equation}
Here and following, we write $f(x)\sim g(x)$ as $x\rar\infty$ to mean that $f(x)/g(x)$ approaches $1$ as $x\rar\infty$. We also use the notation $f(x)=O(g(x))$, or equivalenty, $f(x)\ll g(x)$, as $x\rar\infty$ for when there exist $M,N>0$ such that $|f(x)|\leq Mg(x)$ for all $x\geq N$. The expression $f(x)=O_{c}(g(x))$, or equivalently $f(x)\ll_cg(x)$, means the same thing, though in this case the constants $M$ and $N$ depend on $c$. In Section \ref{ssec:reduce}, we reduce the proof of \eqref{asympt2} to showing a similar asymptotic formula for sums of certain analytic functions over the periodic points of the Gauss map. From that point we adapt the work of Pollicott \cite{P} on a Ruelle-Perron-Frobenius operator of the Gauss map. We introduce this operator in Section \ref{ssec:transfer} and establish its nuclearity and analyticity with respect to certain parameters in Section \ref{ssec:Lnuclear}. Then in Section \ref{ssec:eta}, we utilize the Fredholm determinants of the operator to construct a certain $\eta$ function which is the Laplace transform of an approximation $\tilde{S}_f$ to a sum of the form
\[S_f(T):=\sum_{\omega\in Q_G(T)}f(\omega).\]
We then calculate the residue of a pole of the $\eta$ function and apply the Wiener-Ikehara Tauberian theorem to determine the asymptotic growth rate of $\tilde{S}_f$. We conclude the proof in Section \ref{ssec:equivalent} by showing that $\tilde{S}_f$ is in fact asymptotically equivalent to $S_f$, which uses a fact due to Kelmer \cite[Theorem 3]{Ke} that the sum of $f$ over the odd periodic continued fractions grows asymptotically slower than that over the even.

\subsection{Reduction to equidistribution of $Q_G(T)$}\label{ssec:reduce}
First, let $f\in C([0,1])$, and assume without loss of generality that $f$ is real valued and nonnegative. Then let $g:[0,1]\rar\R$ be defined by $g(x)=xf(x)$. Notice that for any $T>0$, we have
\[\sum_{\omega\in Q_F(T)}\omega f(\omega)=\sum_{[\overline{a_1,\ldots,a_{2n}}]\in Q_G(T)}\sum_{k=0}^{a_1-1}g([a_1-k,\overline{a_2,\ldots,a_{2n},a_1}])
=\sum_{\omega\in Q_G(T)}\sum_{k=0}^{\lfloor\omega^{-1}\rfloor-1}(g\circ F^k)(\omega).\]
So understanding the sum of $g$ over $Q_F(T)$ is equivalent to understanding the sum of the function $\bar{g}:(0,1]\rar\R$ defined by
\[\bar{g}(x)=\sum_{k=0}^{\lfloor x^{-1}\rfloor-1}(g\circ F^k)(x)\]
over $Q_G(T)$. A straightforward calculation reveals that
\begin{align*}
\int_{[0,1]}\bar{g}\,d\nu=\frac{1}{\log2}\int_0^1f(x)\,dx.
\end{align*}
Thus the asymptotic formula \eqref{asympt2} is equivalent to
\[\sum_{\omega\in Q_G(T)}\bar{g}(\omega)\sim\left(\frac{3\log2}{\pi^2}\int_{[0,1]}\bar{g}\,d\nu\right)e^T.\qquad(T\rar\infty)\]
If $\bar{g}$ had an extension to a function in $C([0,1])$, this asymptotic formula would follow from the work of Pollicott. However, in general, $\bar{g}(x)$ has discontinuities at the points $x\in\{n^{-1}:n\in\N,n\geq2\}$ and can grow without bound as $x\rar0^+$, and we must take these facts into account. On the other hand, the following lemma essentially shows that we can assume that $\bar{g}$ is an analytic function of a particular form.

\begin{lemma}\label{L1}
The function $\bar{g}(x)=\sum_{k=0}^{\lfloor x^{-1}\rfloor}(g\circ F^k)(x)$ can be approximated arbitrarily closely in $L^1(\nu)$ from above and below by functions of the form $p(x)=p_1(x)(1-\log x)$, where $p_1$ is a polynomial.
\end{lemma}

\begin{proof}
By the bounds $(2\log2)^{-1}\leq d\nu/dm\leq(\log2)^{-1}$ on the Radon-Nikodym derivative $d\nu/dm=((1+x)\log2)^{-1}$, it suffices to prove the approximation in the $L^1$ norm with respect to the Lebesgue measure. It is clear that $\bar{\bar{g}}(x):=\frac{\bar{g}(x)}{1-\log x}$ is uniformly bounded on $(0,1]$, and it is also continuous except possibly at the points $\{n^{-1}:n\in\N,n\geq2\}$, where there could be jump discontinuities. For any $\epsilon>0$, one can clearly find continuous functions $h_{1,\epsilon},h_{2,\epsilon}:[0,1]\rar\R$ such that $h_{1,\epsilon}\leq\bar{\bar{g}}\leq h_{2,\epsilon}$, $\|h_{1,\epsilon}\|_\infty,\|h_{2,\epsilon}\|_\infty\leq\|\bar{\bar{g}}\|_\infty$,
\[\int_0^1(\bar{\bar{g}}(x)-h_{1,\epsilon}(x))\,dx\leq\epsilon,\quad\text{and}\quad\int_0^1(h_{2,\epsilon}(x)-\bar{\bar{g}}(x))\,dx\leq\epsilon.\]
Then by the Stone-Weierstrass theorem, there exist polynomials $p_{1,\epsilon},p_{2,\epsilon}$ such that $h_{1,\epsilon}(x)-\epsilon\leq p_{1,\epsilon}(x)\leq h_{1,\epsilon}(x)$ and $h_{2,\epsilon}(x)\leq p_{2,\epsilon}(x)\leq h_{2,\epsilon}(x)+\epsilon$ for $x\in[0,1]$. We then have
\begin{align*}
&\int_0^1[\bar{g}(x)-p_{1,\epsilon}(x)(1-\log x)]\,dx=\int_0^1(\bar{\bar{g}}(x)-p_{1,\epsilon}(x))(1-\log x)\,dx\\
&\qquad=\int_0^1(\bar{\bar{g}}(x)-h_{1,\epsilon}(x))(1-\log x)\,dx+\int_0^1(h_{1,\epsilon}(x)-p_{1,\epsilon}(x))(1-\log x)\,dx\\
&\qquad\leq\int_0^\epsilon2\|\bar{\bar{g}}\|_\infty(1-\log x)\,dx+(1+\log\epsilon^{-1})\int_\epsilon^1(\bar{\bar{g}}(x)-h_{1,\epsilon}(x))\,dx+\int_0^1(1-\log x)\,dx\\
&\qquad\leq2\|\bar{\bar{g}}\|_\infty\epsilon(2-\log\epsilon)+\epsilon(1+\log\epsilon^{-1})+2\epsilon.
\end{align*}
The last expression above approaches $0$ as $\epsilon\rar0$. Thus, we can conclude that $\bar{g}(x)$ can be approximated arbitrarily closely in $L^1$ from below by functions of the desired form. Similarly, using $p_{2,\epsilon}$, one can show the approximation of $\bar{g}(x)$ by such functions from above.
\end{proof}

By this approximation result, we have reduced the proof to showing that
\[S_f(T)=\sum_{\omega\in Q_G(T)}f(\omega)\sim\left(\frac{3\log2}{\pi^2}\int_{[0,1]}f\,d\nu\right)e^T.\qquad(T\rar\infty)\]
for functions $f:(0,1]\rar\R$ of the form $f(x)=p(x)(1-\log x)$ where $p$ is a polynomial.

\subsection{The Ruelle-Perron-Frobenius operator}\label{ssec:transfer}
 We now begin following \cite{P}, as well as \cite{F} and \cite{Ke}, in proving the above growth rate by analyzing a Ruelle-Perron-Frobenius operator of the Gauss map. For $r>0$, let $D_r=\{z\in\C:|z-1|<r\}$ and $\Dbar_r$ the closure of $D_r$, and let $a,b\in\R$ be any fixed constants satisfying $1<a<b<3/2$. The Ruelle-Perron-Frobenius operator we define below acts on the disk algebra, which we denote by $\A$ and view as the set of continuous functions on $\Dbar_a$ which are analytic in $D_a$, equipped with the supremum norm $\|\cdot\|_\infty$. Let $f$ be a complex-valued function of the form $f(z)=f_1(z)(1-\log z)$, where $\log z$ is the principal branch of the logarithm and $f_1$ is analytic in an open neighborhood of $\overline{D}_1$ and real valued and positive on $[0,1]$. Then $f$ is analytic in an open neighborhood of $\overline{D}_1\bsl\{0\}$, and letting $K=\|f_1\|_\infty(1+\pi)$, we have $|f(z)|\leq K(1-\log|z|)$ for all $z\in\overline{D}_1\bsl\{0\}$. Then letting $\chi_{s,\omega}(z):=z^{2s}e^{\omega f(z)}$ for $s,\omega\in\C$ (the $2s$ power coming from the principle branch of the logarithm), we define the Ruelle-Perron-Frobenius operator $L_{s,\omega}:\A\rar\A$ by
\[(L_{s,\omega}g)(z)=\sum_{n=1}^\infty g\left(\frac{1}{z+n}\right)\chi_{s,\omega}\left(\frac{1}{z+n}\right).\]
Letting also $\mathcal{U}:=\{(s,\omega)\in\C^2:\Rl(s)>(1+K|\omega|)/2\}$, we see that $L_{s,\omega}$ is a well-defined and bounded operator for $(s,\omega)\in\mathcal{U}$ (with $s=\sigma+it$ for $\sigma,t\in\R$) by the calculation
\begin{align}
\sum_{n=1}^\infty\left|g\left(\frac{1}{z+n}\right)\chi_{s,\omega}\left(\frac{1}{z+n}\right)\right|&\leq\|g\|_\infty\sum_{n=1}^\infty\left|\frac{1}{(z+n)^{2s}}\right|\left|e^{\omega f((z+n)^{-1})}\right|\nonumber\\
&\leq\|g\|_\infty\sum_{n=1}^\infty\left(\frac{e^{2\pi t}}{|z+n|^{2\sigma}}\right)\left(e^{K|\omega|}|z+n|^{K|\omega|}\right)\nonumber\\
&\leq\|g\|_\infty\sum_{n=1}^\infty\frac{e^{2\pi t+K|\omega|}}{(n-\frac{1}{2})^{2\sigma-K|\omega|}}.\label{normcalc1}
\end{align}

\subsection{$L_{s,\omega}$ is nuclear of order $0$ and analytic}\label{ssec:Lnuclear}
We now closely follow the arguments of Faivre \cite{F} to show that for $(s,\omega)\in\mathcal{U}$, $L_{s,\omega}$ is a nuclear operator of order $0$, and is analytic in $(s,\omega)$. We begin with nuclearity.

For a given $\epsilon>0$, we wish to find a sequence $\{\Lambda_j\otimes e_j\}_{j=0}^\infty\subseteq\A^*\otimes\A$ such that
\[L_{s,\omega}=\sum_{j=0}^\infty\Lambda_j\otimes e_j,\quad\text{and}\quad\sum_{j=0}^\infty\|\Lambda_j\|^\epsilon\|e_j\|_\infty^\epsilon<\infty,\]
where $\Lambda_j\otimes e_j$ is defined as an operator on $\A$ by $((\Lambda_j\otimes e_j)g)(z)=(\Lambda_jg)e_j(z)$.
Assume $(s,\omega)\in\mathcal{U}$ so that $\sigma>\frac{1+K|\omega|}{2}$, fix $g\in\A$, and for each $n\in\N$ let
\[(L_{n,s,\omega}g)(z):=h_{g,n,s,\omega}(z):=g\left(\frac{1}{z+n}\right)\chi_{s,\omega}\left(\frac{1}{z+n}\right).\]
It is easy to see that $h_{g,n,s,\omega}$ is an analytic function on $D_{3/2}$, and so for $z\in\Dbar_a\subseteq D_b$,
\begin{align*}
h_{g,n,s,\omega}(z)=\sum_{j=0}^\infty\frac{h_{g,n,s,\omega}^{(j)}(1)}{j!}(z-1)^j.
\end{align*}
Define the element $\Lambda_{n,j,s,\omega}\in\A^*$ by
\[\Lambda_{n,j,s,\omega}g:=\frac{h_{g,n,s,\omega}^{(j)}(1)}{j!}=\frac{1}{2\pi i}\int\limits_{|\zeta-1|=b}\frac{h_{g,n,s,\omega}(\zeta)}{(\zeta-1)^{j+1}}d\zeta;\]
the latter equality follows from Cauchy's formula. Also define $e_j\in\A$ by $e_j(z)=(z-1)^j$ so that
\[h_{g,n,s,\omega}=\sum_{j=0}^\infty(\Lambda_{n,j,s,\omega}g)e_j.\]
By the calculation \eqref{normcalc1}, we have
\[\|h_{g,n,s,\omega}\|_\infty\leq\frac{\|g\|_\infty e^{2\pi t+K|\omega|}}{(n-\frac{1}{2})^{2\sigma-K|\omega|}},\]
and therefore
\begin{align*}
|\Lambda_{n,j,s,\omega}g|\leq\frac{\|h_{g,n,s,\omega}\|_\infty}{2\pi}\int\limits_{|\zeta-1|=b}\frac{1}{|\zeta-1|^{j+1}}|d\zeta|=\frac{\|g\|_\infty e^{2\pi t+K|\omega|}}{b^j(n-\frac{1}{2})^{2\sigma-K|\omega|}}.
\end{align*}
This implies that
\[\sum_{n=1}^\infty\|\Lambda_{n,j,s,\omega}\|\leq\sum_{n=1}^\infty\frac{e^{2\pi t+K|\omega|}}{b^j(n-\frac{1}{2})^{2\sigma-K|\omega|}},\]
which is a finite quantity, and we denote its product with $b^j$ by $\kappa(s,\omega)$. Hence
\[\Lambda_{j,s,\omega}:=\sum_{n=1}^\infty\Lambda_{n,j,s,\omega}\]
is a well defined element of $\A^*$.

Next, notice that
\[\sum_{j=0}^\infty\sum_{n=1}^\infty\|\Lambda_{n,j,s,\omega}\|\|e_j\|_\infty
\leq\sum_{j=0}^\infty\sum_{n=1}^\infty\frac{e^{2\pi t+K|\omega|}}{b^j(n-\frac{1}{2})^{2\sigma-K|\omega|}}a^j
=\kappa(s,\omega)\sum_{j=0}^\infty\left(\frac{a}{b}\right)^j
=\frac{\kappa(s,\omega)}{1-(a/b)}<\infty.\]
As a result, we have
\begin{align*}
L_{s,\omega}&=\sum_{n=1}^\infty L_{n,s,\omega}=\sum_{n=1}^\infty\sum_{j=0}^\infty\Lambda_{n,j,s,\omega}\otimes e_j=\sum_{j=0}^\infty\Lambda_{j,s,\omega}\otimes e_j.
\end{align*}
Finally, note that for all $\epsilon>0$,
\[\sum_{j=0}^\infty\|\Lambda_{j,s,\omega}\|^\epsilon\|e_j\|_\infty^\epsilon\leq\sum_{j=0}^\infty\frac{\kappa(s,\omega)^\epsilon}{b^{\epsilon j}}a^{\epsilon j}=\frac{\kappa(s,\omega)^\epsilon}{1-(a/b)^\epsilon}<\infty.\]
This completes the proof that $L_{s,\omega}$ is nuclear of order $0$.

We now prove the analyticity of $L_{s,\omega}$, that is, that $(s,\omega)\mapsto L_{s,\omega}$ is analytic as a map from $\mathcal{U}$ to the order $\epsilon$ nuclear operators on $\A$ for any $\epsilon>0$. We only show the proof for the analyticity with respect to $s$ since the proof for $\omega$ is essentially the same. Throughout, we fix $(s_0,\omega_0)\in\mathcal{U}$ at which we prove the differentiability of $L_{s,\omega_0}$ with respect to $s$. Also, let $s_0=\sigma_0+it_0$ for $\sigma_0,t_0\in\R$.

For our first step, we fix $n,j\in\N\cup\{0\}$ with $n\geq1$, and show that $(s,\omega)\mapsto\Lambda_{n,j,s,\omega}$ is differentiable (as a map from $\mathcal{U}$ to $\A^*$) at $(s_0,\omega_0)$ with respect to $s$. Define the element $\Theta_{n,j,s,\omega}\in\A^*$ by
\[\Theta_{n,j,s,\omega}g=\frac{1}{2\pi i}\int\limits_{|\zeta-1|=b}\frac{-2h_{g,n,s,\omega}(\zeta)\log(\zeta+n)}{(\zeta-1)^{j+1}}d\zeta.\]
We aim to show that $\Theta_{n,j,s,\omega}=\frac{\partial}{\partial s}\Lambda_{n,j,s,\omega}$. For $(s,\omega_0)\in\mathcal{U}$ and $g\in\A$, we have
\begin{align}
&\left|\left(\frac{\Lambda_{n,j,s,\omega_0}-\Lambda_{n,j,s_0,\omega_0}}{s-s_0}-\Theta_{n,j,s_0,\omega_0}\right)g\right|\nonumber\\
&\qquad=\left|\frac{1}{2\pi i}\int\limits_{|\zeta-1|=b}\frac{g((\zeta+n)^{-1})e^{\omega_0f((\zeta+n)^{-1})}}{(\zeta-1)^{j+1}(\zeta+n)^{2s_0}}
\left(\frac{(\zeta+n)^{-2(s-s_0)}-1}{s-s_0}+2\log(\zeta+n)\right)d\zeta\vphantom{\int\limits_{|\zeta-1|=b}}\right|.\label{appeq1}
\end{align}
Notice that for a function $\psi$ that is analytic in an open neighborhood $N$ of $0$, and for $h\in\C$ such that the straight line segment $[0,h]$ connecting $0$ and $h$ is in $N$, we have
\begin{align}
|\psi(h)-\psi(0)-h\psi'(0)|&=\left|\int_{[0,h]}\psi(u)\,du-h\psi'(0)\right|
=\left|\int_{[0,h]}(\psi'(u)-\psi'(0))du\right|\nonumber\\
&=\left|\int_{[0,h]}\frac{\psi'(u)-\psi'(0)}{u}u\,du\right|
\leq\int_{[0,h]}\max_{u'\in[0,u]}|\psi''(u')||u||du|\nonumber\\
&\leq\left(\max_{u\in[0,h]}|\psi''(u)|\right)\int_0^{|h|}x\,dx
=\frac{|h|^2}{2}\max_{u\in[0,h]}|\psi''(u)|.\label{appeq2}
\end{align}
This implies that
\begin{align*}
&\left|\frac{(\zeta+n)^{-2(s-s_0)}-1}{s-s_0}+2\log(\zeta+n)\right|\leq\frac{|s-s_0|}{2}\max_{u\in[0,s-s_0]}\left|\frac{4(\log(\zeta+n))^2}{(\zeta+n)^{2u}}\right|\\
&\qquad\qquad\leq2|s-s_0|(\pi+\log(n+1+b))^2e^{2\pi|\Ig(s-s_0)|}(n+1+b)^{2|\Rl(s-s_0)|}.
\end{align*}
Hence, \eqref{appeq1} is at most
\[\frac{2(\pi+\log(n+1+b))^2e^{2\pi(t_0+|\Ig(s-s_0)|)+K|\omega_0|}(n+1+b)^{2|\Rl(s-s_0)|}}{b^j(n-\frac{1}{2})^{2\sigma_0-K|\omega_0|}}|s-s_0|\|g\|_\infty.\]
For ease of notation, we let $\Delta_n(s,\omega,s_0,\omega_0)>0$ be the constant such that the bound above equals $\Delta_n(s,\omega,s_0,\omega_0)b^{-j}|s-s_0|\|g\|_\infty$. Then
\begin{align*}
&\left\|\frac{\Lambda_{n,j,s,\omega_0}-\Lambda_{n,j,s_0,\omega_0}}{s-s_0}-\Theta_{n,j,s_0,\omega_0}\right\|\leq\frac{\Delta_n(s,\omega,s_0,\omega_0)}{b^j}|s-s_0|,
\end{align*}
which approaches $0$ as $s\rar s_0$. This proves that $\Theta_{n,j,s,\omega}=\frac{\partial}{\partial s}\Lambda_{n,j,s,\omega}$.

Next, notice that
\begin{align*}
\sum_{n=1}^\infty\|\Theta_{n,j,s,\omega}\|\leq\sum_{n=1}^\infty\frac{e^{2\pi t+K|\omega|}2(\pi+\log(n+1+b))}{b^j(n-\frac{1}{2})^{2\sigma-K|\omega|}},
\end{align*}
which is finite for $(s,\omega)\in\mathcal{U}$; and hence
\[\Theta_{j,s,\omega}:=\sum_{n=1}^\infty\Theta_{n,j,s,\omega}\]
is a well defined element of $\A^*$ for all $j\in\N\cup\{0\}$. We then have
\begin{align*}
&\left\|\frac{\Lambda_{j,s,\omega_0}-\Lambda_{j,s_0,\omega_0}}{s-s_0}-\Theta_{j,s_0,\omega_0}\right\|\leq\sum_{n=1}^\infty\frac{\Delta_n(s,\omega,s_0,\omega_0)}{b^j}|s-s_0|,
\end{align*}
which is finite as long as $s$ is close enough to $s_0$ so that $2\sigma_0-K|\omega_0|-2|\Rl(s-s_0)|>1$. Such $s$ comprise an open neighborhood of $s_0$ since $2\sigma_0-K|\omega_0|>1$. It is thus clear that as $s\rar s_0$, the left side above approaches $0$. This proves that $\frac{\partial}{\partial s}\Lambda_{j,s,\omega}$ exists and equals $\Theta_{j,s,\omega}$.

Next fix $\epsilon>0$. Then notice that
\begin{align*}
\sum_{j=0}^\infty\|\Theta_{j,s,\omega}\|^\epsilon\|e_j\|^\epsilon&\leq\sum_{j=0}^\infty\left(\sum_{n=1}^\infty\frac{e^{2\pi t+K|\omega|}2(\pi+\log(n+1+b))}{b^j(n-\frac{1}{2})^{2\sigma-K|\omega|}}\right)^\epsilon a^{j\epsilon}\\
&=\left(\sum_{n=1}^\infty\frac{e^{2\pi t+K|\omega|}2(\pi+\log(n+1+b))}{(n-\frac{1}{2})^{2\sigma-K|\omega|}}\right)^\epsilon\frac{1}{1-(a/b)^\epsilon}<\infty,
\end{align*}
for $(s,\omega)\in\mathcal{U}$, implying that
\[\Theta_{s,\omega}:=\sum_{j=0}^\infty\Theta_{j,s,\omega}\otimes e_j\]
is a nuclear operator of order $\epsilon$. We also have
\begin{align*}
\sum_{j=0}^\infty\left\|\frac{\Lambda_{j,s,\omega_0}-\Lambda_{j,s_0,\omega_0}}{s-s_0}-\Theta_{j,s_0,\omega_0}\right\|^\epsilon\|e_j\|^\epsilon&
\leq\sum_{j=0}^\infty\left(\sum_{n=1}^\infty\frac{\Delta_n(s,\omega,s_0,\omega_0)}{b^j}|s-s_0|\right)^\epsilon a^{j\epsilon},\\
&=\frac{|s-s_0|^\epsilon}{1-(a/b)^\epsilon}\left(\sum_{n=1}^\infty\Delta_n(s,\omega,s_0,\omega_0)\right)^\epsilon,
\end{align*}
which again is finite as long as $2\sigma_0-K|\omega_0|-2|\Rl(s-s_0)|>1$, and approaches $0$ as $s\rar s_0$. This means that
\[\frac{L_{s,\omega_0}-L_{s_0,\omega_0}}{s-s_0}-\Theta_{s_0,\omega_0}=\sum_{j=0}^\infty\left(\frac{\Lambda_{j,s,\omega_0}-\Lambda_{j,s_0,\omega_0}}{s-s_0}-\Theta_{j,s_0,\omega_0}\right)\otimes e_j\]
approaches $0$ in the norm on order $\epsilon$ nuclear operators. This completes the proof that $(s,\omega)\mapsto L_{s,\omega}$ is analytic with respect to $s$ as a map to the order $\epsilon$ nuclear operators ($\forall\epsilon>0$) and $\frac{\partial}{\partial s}L_{s,\omega}=\Theta_{s,\omega}$. As mentioned above, one can similarly prove that $L_{s,\omega}$ is analytic with respect to $\omega$ as well.

\subsection{The $\eta$-function}\label{ssec:eta}
By the work of Grothendieck \cite{G1,G2} on the theory of Fredholm determinants of nuclear operators on Banach spaces, the functions
\[Z_\pm(s,\omega):=\det(I\pm L_{s,\omega})=\prod_{\lambda_{s,\omega}\in\spec(L_{s,\omega})}(1\pm\lambda_{s,\omega}),\]
where $\spec(L_{s,\omega})$ is the set of eigenvalues of $L_{s,\omega}$ (counted with multiplicity), are well-defined and analytic on $\mathcal{U}$. Furthermore, the product over the eigenvalues converges absolutely. By \cite[Proposition 3.4.]{F}, if $\Rl(s)>1$ or $s=1+it$ with $t\in\R\bsl\{0\}$, then the spectral radius of $L_{s,0}$ is less than $1$, and hence there is an open neighborhood $\mathcal{V}$ of $\{(s,0)\in\C^2:\Rl(s)\geq1,s\neq1\}$ in $\mathcal{U}$ such that the spectral radius of $L_{s,\omega}$ is less than $1$ for all $(s,\omega)\in\mathcal{V}$. Thus, for $(s,\omega)\in\mathcal{V}$, $Z_\pm(s,\omega)\neq0$ and
\begin{align*}
Z_\pm(s,\omega)&=\exp\left(\sum_{\lambda_{s,\omega}\in\sigma(L_{s,\omega})}\log(1\pm\lambda_{s,\omega})\right)
=\exp\left(\sum_{\lambda_{s,\omega}\in\sigma(L_{s,\omega})}\sum_{n=1}^\infty\frac{(-1)^{n-1}}{n}(\pm\lambda_{s,\omega})^n\right)\\
&=\exp\left(-\sum_{n=1}^\infty\frac{(\mp1)^n}{n}\Tr(L_{s,\omega}^n)\right).
\end{align*}
Assuming that $\Rl(s)>1$ and following \cite{M} (see also \cite[Theorem 3.3]{F} for a detailed argument which can readily be applied to our situation), we have
\[\Tr(L_{s,\omega}^n)=\sum_{a_1,\ldots,a_n=1}^\infty\frac{\prod_{j=1}^n\chi_{s,\omega}(G^j([\overline{a_1,\ldots,a_n}]))}{1-(-1)^n\prod_{j=1}^nG^j([\overline{a_1,\ldots,a_n}])^2}.\]
It follows that if $\zeta_\pm(s,\omega):=Z_\pm(s+1,\omega)/Z_\mp(s,\omega)$, then
\begin{align*}
\zeta_+(s,\omega)&=\exp\left(\sum_{n=1}^\infty\frac{1}{n}\sum_{a_1,\ldots,a_n=1}^\infty\prod_{j=1}^n\chi_{s,\omega}(G^j([\overline{a_1,\ldots,a_n}]))\right),\\
\zeta_-(s,\omega)&=\exp\left(\sum_{n=1}^\infty\frac{(-1)^n}{n}\sum_{a_1,\ldots,a_n=1}^\infty\prod_{j=1}^n\chi_{s,\omega}([G^j(\overline{a_1,\ldots,a_n}]))\right),
\end{align*}
and hence
\begin{align*}
\eta(s,\omega)&:=\frac{1}{2}\log(\zeta_+(s,\omega)\zeta_-(s,\omega))=\sum_{n=1}^\infty\frac{1}{2n}\sum_{a_1,\ldots,a_{2n}=1}^\infty\prod_{j=1}^{2n}\chi_{s,\omega}(G^j[\overline{a_1,\ldots,a_{2n}}]).
\end{align*}
Taking the derivative of $\eta$ with respect to $\omega$ and setting $\omega=0$ yields the function
\begin{align*}
\hat{\eta}(s)&:=\sum_{n=1}^\infty\frac{1}{2n}\sum_{a_1,\ldots,a_{2n}=1}^\infty\left(\sum_{j=1}^{2n}f(G^j([\overline{a_1,\ldots,a_{2n}}]))\right)e^{-s\ell(a_1,\ldots,a_{2n})}
\end{align*}

For a given tuple $a=(a_1,\ldots,a_n)\in\N^n$ of any length, we define $\per(a)$ be the length of the minimal period in the periodic continued fraction $[\overline{a}]$, which is the number of distinct tuples of length $n$ that one can cyclically permute to obtain $a$. Then letting $n\in\N$ be fixed, the term $f([\overline{a'}])e^{-s\ell(a')}$ corresponding to a given $a'\in\N^{2n}$ appears $\per(a')$ times in the sum over $a_1,\ldots,a_{2n}$ and $j$ in the definition of $\hat{\eta}$. So we can combine terms to get
\[\hat{\eta}(s)=\sum_{n=1}^\infty\sum_{a\in\N^{2n}}\frac{\per(a)}{2n}f([\overline{a}])e^{-s\ell(a)}.\]
This establishes $\hat{\eta}$ as the Laplace transform
\[\int_{0}^\infty e^{-st}\,d\tilde{S}_f(t)\]
of the function
\[\tilde{S}_f(T)=\sum_{n=1}^\infty\sum_{\substack{a\in\N^{2n}\\\ell(a)\leq T}}\frac{\per(a)}{2n}f([\overline{a}]).\]

Recall that $\hat{\eta}(s)$ is the $\omega$-derivative of $\eta(s,\omega)$ at $\omega=0$, and hence $\hat{\eta}$ extends analytically to a neighborhood of $\{s\in\C:\Rl(s)\geq1\}\bsl\{1\}$. In this section, we see that $s=1$ is a simple pole of $\hat{\eta}$ and calculate its corresponding residue, which allows us to determine the asymptotic growth rate of $\tilde{S}_f$.

A fact due to Wirsing \cite{W} is that $1$ is the maximal eigenvalue of $L_{1,0}$, and all other eigenvalues have modulus less than $0.31$. By analytic perturbation theory (see \cite{Ka}), there is a neighborhood $\mathcal{W}$ of $(1,0)$ such that for $(s,\omega)\in\mathcal{W}$, the maximal eigenvalue $\lambda_1(s,\omega)$ of $L_{s,\omega}$ is analytic in $(s,\omega)$ and the lesser eigenvalues of $L_{s,\omega}$ are of modulus less than $1$. So for $(s,\omega)\in\mathcal{V}\cap\mathcal{W}$,
\[\eta(s,\omega)=-\frac{1}{2}\log(1-\lambda_1(s,\omega)^2)+\Phi(s,\omega),\]
where $\Phi$ is analytic in $\mathcal{W}$. Hence, for $(s,0)\in\mathcal{V}\cap\mathcal{W}$,
\[\hat{\eta}(s)=\frac{\lambda_1(s,0)}{1-(\lambda_1(s,0))^2}\frac{\partial\lambda_1}{\partial\omega}(s,0)+\frac{\partial\Phi}{\partial\omega}(s,0),\]
which then analytically extends the domain of $\hat{\eta}(s)$ to the set of $s\in\C$ such that $(s,0)\in\mathcal{W}$ and $\lambda_1(s,0)\neq1$. So if $\frac{\partial\lambda_1}{\partial s}(1,0)\neq0$, then $\hat{\eta}(s)$ has a simple pole at $s=1$ with residue
\[\frac{-\frac{\partial\lambda_1}{\partial\omega}(1,0)}{2\frac{\partial\lambda_1}{\partial s}(1,0)}.\]
From the proof of \cite[Proposition 2]{P} and \cite[Theorem 3.6]{F}, it follows that $-\frac{\partial\lambda_1}{\partial s}(1,0)$ is the entropy of the Gauss map, which is $\frac{\pi^2}{6\log2}$; and \cite{P} also establishes that $\frac{\partial\lambda_1}{\partial\omega}(1,0)=\int_{[0,1]}f\,d\nu$.

We prove the latter assertion in detail, taking into account the fact that $f(z)$ can have an asymptote at $z=0$. Note first that if $\omega\in\R$, then it is clear that $\Tr(L_{1,\omega}^n)>0$ for all $n\in\N$, and hence $\lambda_1(1,\omega)$ is real and positive. Therefore, $\lambda_1(1,\omega)$ can be calculated as the spectral radius of $L_{1,\omega}$ for $\omega\in\R$. So letting $\omega\in\R$ and following \cite[Theorem 3.6]{F}, we have
\begin{align*}
&\log\lambda_1(1,\omega)=\lim_{n\rar\infty}\frac{1}{n}\log\|L_{1,\omega}^n\|\\
&\qquad\geq\lim_{n\rar\infty}\frac{1}{n}\log\left(\sum_{a_1,\ldots,a_n=1}^\infty\left(\prod_{j=1}^n[a_j,\ldots,a_n]^2\right)\exp\left(\omega\sum_{j=1}^nf([a_j,\ldots,a_n])\right)\right)
\end{align*}
By the properties \eqref{prodprop} and \eqref{cfgapmeas}, we have
\[\prod_{j=1}^n[a_j,\ldots,a_n]^2=\frac{1}{q_{1,n}^2}\geq\frac{1}{q_{1,n}(q_{1,n}+q_{1,n-1})}=m(\llbr a_1,\ldots,a_n\rrbr).\]
This and the concavity of $\log$ yield
\begin{align*}
\log\lambda_1(1,\omega)&\geq\lim_{n\rar\infty}\frac{1}{n}\log \left(\sum_{a_1,\ldots,a_n=1}^\infty m(\llbracket a_1,\ldots,a_n\rrbracket)\exp\left(\omega\sum_{j=1}^nf([a_j,\ldots,a_n])\right)\right)\\
&\geq\omega\lim_{n\rar\infty}\frac{1}{n}\sum_{j=1}^n\sum_{a_1,\ldots,a_n=1}^\infty m(\llbracket a_1,\ldots,a_n\rrbracket)\cdot f([a_j,\ldots,a_n]).
\end{align*}
Next, notice that
\begin{align*}
f([a_j,\ldots,a_n])&=\frac{1}{m(\llbracket a_1,\ldots,a_n\rrbracket)}\int_{\llbracket a_1,\ldots,a_n\rrbracket}(f\circ G^{j-1})\,dm
+O\left(\left\|f'\big|_{\llbracket a_j,\ldots,a_n\rrbracket}\right\|_\infty\cdot m(\llbracket a_j,\ldots,a_n\rrbracket)\right);
\end{align*}
and thus
\begin{align}
\log\lambda_1(1,\omega)
&\geq\omega\lim_{n\rar\infty}\left(\int_{[0,1]}\frac{1}{n}\sum_{j=1}^n(f\circ G^{j-1})\,dm+O\left(\sup_{a_1,\ldots,a_n}\frac{1}{n}\sum_{j=1}^n\frac{\|f'|_{\llbracket a_j,\ldots,a_n\rrbracket}\|_\infty}{q_{j,n}(q_{j,n}+q_{j,n-1})}\right)\right)\nonumber\\
&=\omega\int_0^1f\,d\nu+O\left(\omega\lim_{n\rar\infty}\frac{1}{n}\sup_{a_1,\ldots,a_n}\sum_{j=1}^n\frac{\|f'|_{\llbracket a_j,\ldots,a_n\rrbracket}\|_\infty}{q_{j,n}(q_{j,n}+q_{j,n-1})}\right),\label{asympt3}
\end{align}
where we used the von Neumann ergodic theorem \cite{vN} and the ergodicity of $G$ to derive the equality.

We wish to prove that the error term in \eqref{asympt3} is equal to $0$. To do so, note that by the definition of $f$, there exists $C>0$ such that $|f'(z)|\leq\frac{C}{|z|}$ for $z\in\overline{D}_1\bsl\{0\}$, and so $\|f'|_{\llbracket a_j,\ldots,a_n\rrbracket}\|_\infty\leq C(a_j+1)$. From Section \ref{ssec:cfracs}, we have the property $q_{j,n}=a_jq_{j+1,n}+q_{j+2,n}$, and as a result,
\[\sum_{j=1}^n\frac{\|f'|_{\llbracket a_j,\ldots,a_n\rrbracket}\|_\infty}{q_{j,n}(q_{j,n}+q_{j,n-1})}\leq\sum_{j=1}^n\frac{2C}{q_{j,n}}.\]
Now for any $a_1,\ldots,a_n\in\N$, $q_{j,n}\geq F_{n-j+1}$, where $\{F_j\}_{j=1}^\infty$ is the Fibonacci sequence with $F_1=F_2=1$; and since the Fibonacci sequence grows at an exponential rate, we can say that
\[\sup_{n}\sup_{a_1,\ldots,a_n}\sum_{j=1}^n\frac{2C}{q_{j,n}}\leq\sum_{j=1}^\infty\frac{2C}{F_j}<\infty.\]
This proves that
\[\lim_{n\rar\infty}\frac{1}{n}\sup_{a_1,\ldots,a_n}\sum_{j=1}^n\frac{\|f'|_{\llbracket a_j,\ldots,a_n\rrbracket}\|_\infty}{q_{j,n}(q_{j,n}+q_{j,n-1})}=0,\]
and hence
\[\log \lambda_1(1,\omega)-\omega\int_{[0,1]}f\,d\nu\geq0\]
for $(1,\omega)\in\mathcal{U}$. Since $\mathcal{U}$ contains an open neighborhood of $(1,0)$, and thus the above holds for $\omega$ in an open neighborhood of $0$, and the expression on the left is equal to $0$ when $\omega=0$, the expression's derivative at $\omega=0$ must vanish. So
\[\frac{\frac{\partial\lambda_1}{\partial\omega}(1,0)}{\lambda_1(1,0)}-\int_{[0,1]}f\,d\nu=0;\]
and since $\lambda_1(1,0)=1$, we have
\[\frac{\partial\lambda_1}{\partial\omega}(1,0)=\int_{[0,1]}f\,d\nu.\]

We have therefore established that the function $\hat{\eta}(s)$ has a pole at $s=1$ with residue $\frac{3\log2}{\pi^2}\int_0^1f\,d\nu$.
Then by the Wiener-Ikehara tauberian theorem \cite[Section III, Theorem 4.2]{Ko}, we have
\[\tilde{S}_f(T)\sim\left(\frac{3\log2}{\pi^2}\int_0^1f\,d\nu\right)e^T.\qquad(T\rar\infty)\]

\subsection{$S_f(T)$ and $\tilde{S}_f(T)$ are asymptotically equivalent}\label{ssec:equivalent}
We first rewrite the sum defining $\tilde{S}_f(T)$ in terms of the periodic points of $G$. For this we need to distinguish between periodic continued fractions of even and odd period by defining the sets
\begin{align*}
Q_{G,\text{even}}(T)&=\{[\overline{a}]:a\in\N^{2n},\,n\in\N,\,\per(a)=2n,\,\ell(a)\leq T\},\\
Q_{G,\text{odd}}(T)&=\{[\overline{a}]:a\in\N^{2n},\,n\in\N\text{ is odd},\,\per(a)=n,\,\ell(a)\leq T\}.
\end{align*}
Now let $a\in\N^{2n}$, with $\ell(a)\leq T$, be a tuple represented in the sum defining $\tilde{S}_f(T)$. Then by the definition of $\per(a)$, $a$ is the concatenation of the tuple $(a_1,\ldots,a_{\per(a)})$ with itself $2n/\per(a)$ times. If $\per(a)$ is even, let $k=2n/\per(a)$, and if $\per(a)$ is odd, let $k=n/\per(a)$. Then letting $\omega=[\overline{a}]$, we have $\ell(a)=k\ell(\omega)$ whether $\per(a)$ is even or odd. Thus, to a given tuple $a$ in the sum defining $\tilde{S}_f(T)$, we have associated elements $\omega\in Q_G$ and $k\in\N$ such that $\ell(a)=k\ell(\omega)$. So we can rewrite $\tilde{S}_f(T)$ as follows. First define
\begin{equation}\label{Sbar}
\bar{S}_f(T)=\sum_{\omega\in Q_{G,\text{even}}(T)}f(\omega)+\frac{1}{2}\sum_{\omega\in Q_{G,\text{odd}}(T)}f(\omega).
\end{equation}
We then have
\[\tilde{S}_f(T)=\sum_{k=1}^{\lfloor T/\ell_0\rfloor}\frac{1}{k}\bar{S}_f(T/k),\]
where $\ell_0$ is the length of the shortest closed geodesic in $T_1\mathcal{M}$.  Noting that $f$ is real valued and positive on $(0,1]$, we see that $\bar{S}_f(T)\ll\tilde{S}_f(T)\ll e^T$ as $T\rar\infty$, and hence
\[\sum_{k=2}^{\lfloor T/\ell_0\rfloor}\bar{S}_f(T/k)\ll Te^{T/2}.\qquad(T\rar\infty)\]
This yields
\begin{equation}\label{asymptotic1}
\bar{S}_f(T)\sim\tilde{S}_f(T)\sim\left(\frac{3\log2}{\pi^2}\int_0^1f\,d\nu\right)e^T.
\end{equation}

To complete the proof, by \eqref{Sbar} it suffices to establish that
\[\sum_{\omega\in Q_{G,\text{odd}}(T)}f(\omega)\ll e^{T/2}.\]
To show this, note that
\[\frac{\partial}{\partial\omega}\left(\log\zeta_+(s,\omega)\right)\Big|_{\omega=0}=\sum_{n=1}^\infty\sum_{a\in\N^n}\frac{\per(a)}{n}f([\overline{a}])e^{-s\ell(a)}\]
is the Laplace transform of the function
\[\bar{\bar{S}}_f(T)=\sum_{n=1}^\infty\sum_{\substack{a\in\N^n\\\ell(a)\leq T}}\frac{\per(a)}{n}f([\overline{a}]),\]
and has a simple pole at $s=1$, though is otherwise analytic in an open neighborhood of $\{s\in\C:\Rl(s)\geq1\}$. So by the Wiener-Ikehara Tauberian theorem, $\bar{\bar{S}}_f(T)\ll e^T$. Note also that
\[\sum_{\omega\in Q_{G,\text{odd}}(T)}f(\omega)\leq\bar{\bar{S}}\left(\frac{T}{2}\right),\]
since if $\omega=[\overline{a_1,\ldots,a_n}]$, where $n$ is odd and the minimal period length in the continued fraction expansion of $\omega$, then $\ell(\omega)=2\ell(a_1,\ldots,a_n)$ so that the inequalities $\ell(\omega)\leq T$ and $\ell(a_1,\ldots,a_n)\leq T/2$ are equivalent. We thus have
\[\sum_{\omega\in Q_{G,\text{odd}}(T)}f(\omega)\ll e^{T/2}.\]
(One can likely adapt the work of Kelmer \cite[Theorem 3]{Ke} to prove a more precise estimate.) This imples that
\[S_f(T)\sim\bar{S}_f(T)\sim\left(\frac{3\log2}{\pi^2}\int_0^1f\,d\nu\right)e^T,\]
and therefore the proof of Theorem \ref{T1} is complete.

\begin{acknowledgements}
I thank my advisor Florin Boca as well as Claire Merriman for many helpful discussions which inspired this paper. I particularly thank Merriman for making me aware of the characterization of the periodic points of $F$. I also thank the referee for helpful suggestions that improved the presentation of this paper. I also acknowledge support from Department of Education Grant P200A090062, ``University of Illinois GAANN Mathematics Fellowship Project.''
\end{acknowledgements}

\pagebreak

\appendix


\section{Counting periodic points of the Farey map through a number theoretical method}

The aim of this appendix is to provide a precise estimate of the cardinality of the set
\begin{equation*}
Q_{\tilde{F}} (x,y;T) =\{ (\omega,\tilde{\omega})\in Q_{\tilde{F}} (T): \omega \geq x, \tilde{\omega} \geq y\}
\end{equation*}
of periodic points of the natural extension $\tilde{F}$ of the Farey map, of length $\ell(\omega) \leq T$ and
lying in a region $[x,1]\times [y,1]$. The proof relies on arguments from \cite{KOPS,Bo,Ust}.

\begin{theorem}\label{T1appendix}
For every $x,y\in [0,1]$, $(x,y)\neq (0,0)$, and every $\epsilon >0$ we have
\begin{equation*}
\# Q_{\tilde{F}} (x,y;T) = e^T \log \left( \frac{1}{x+y-xy}\right) +O_\epsilon
\left( \left(\frac{e^{T/2}}{x+y}\right)^{3/2+\epsilon}\right).
\end{equation*}
\end{theorem}

\begin{remark}
Since
\begin{equation*}
\log \left( \frac{1}{x+y-xy}\right) = \int_x^1 \int_y^1 \frac{du\, dv}{(u+v-uv)^2},
\end{equation*}
Theorem \ref{T1appendix} shows that the unweighted periodic points of $\tilde{F}$ are equidistributed with respect to the $\tilde{F}$-invariant
measure $\tilde{\mu}$. The effective estimate also allows $x$ and $y$ to converge to $0$ in a controlled way as $T\rightarrow \infty$.
\end{remark}

\begin{corollary}
The periodic points of the Farey map are equidistributed with respect to the $F$-invariant measure $\mu$, that is,
for every $x\in (0,1]$,
\begin{equation*}
\# \{ \omega \in Q_F(T): \ell(\omega) \leq T,\omega \geq x\}  =e^T
\log \bigg( \frac{1}{x}\bigg) +O_\epsilon
\bigg( \Big( \frac{e^{T/2}}{x}\Big)^{3/2+\epsilon} \bigg).
\end{equation*}
\end{corollary}

It is well known that products of positive powers of the matrices $A=\left( \begin{smallmatrix} 1 & 0 \\ 1 & 1 \end{smallmatrix}\right)$ and $B=\left( \begin{smallmatrix} 1 & 1 \\ 0 & 1 \end{smallmatrix} \right)$ satisfy
\begin{equation}\label{mproducts}
B^{a_1} A^{a_2} \cdots B^{a_{2m-1}} A^{a_{2m}}=\left( \begin{matrix} q_{2m} & q_{2m-1} \\
p_{2m} & p_{2m-1} \end{matrix}\right),\quad
B^{a_1}A^{a_2} \cdots A^{a_{2m}} B^{a_{2m+1}}=\left( \begin{matrix} q_{2m} & q_{2m+1} \\
p_{2m} & p_{2m+1} \end{matrix} \right),
\end{equation}
where $p_k=p_k(a_1,\dots,a_{2m+1})$ and $q_k=q_k(a_1,\ldots,a_{2m+1})$ are defined as in Section \ref{ssec:cfracs}. We refer to the products of the form on the left as even products, and the right as odd products. We associate the even product with the element $\omega=[\overline{a_1,\ldots,a_{2m}}]\in Q_G$. The fractions $p_{2m}/q_{2m}$ and $q_{2m-1}/q_{2m}$ provide good approximations for $\omega$ and respectively $-\bar{\omega}^{-1}$ as follows:
\begin{equation}\label{CFapprox}
\left|\omega-\frac{p_{2m}}{q_{2m}}\right|<\frac{1}{q_{2m}^2},\quad\left|(-\bar{\omega}^{-1})-\frac{q_{2m-1}}{q_{2m}}\right|<\frac{1}{q_{2m}^2}.
\end{equation}
(These inequalities follow easily from properties \eqref{CFprop1} and \eqref{CFinvert}.) Furthermore, the trace $q_{2m}+p_{2m-1}$ of the even matrix product is very close to $e^{\ell(a_1,\ldots,a_{2m})/2}$. Using these properties, Ustinov \cite{Ust} was able to establish an effective equidistribution result for the periodic points of $\tilde{G}$ according to $\tilde{\nu}$ by counting the number of even products with entries satisfying certain conditions. In particular, if we let $\Psi_{\text{ev}}(\alpha,\beta;N)$ be the set of even products in \eqref{mproducts} such that $p_{2m}/q_{2m}\leq\alpha$, $q_{2m-1}/q_{2m}\leq\beta$, and $q_{2m}+p_{2m-1}\leq N$, which is the set
\[\left\{ \left(\begin{matrix} q^\prime & q \\ p^\prime & p \end{matrix} \right) \in \SL_2(\Z) :
0\leq p\leq q,\ 0\leq p'\leq q',\  \frac{p^\prime}{q^\prime} \leq \alpha,\  \frac{q}{q^\prime} \leq \beta,\
 p+q^\prime \leq N \right\},\]
then it follows from \cite[Theorem 1]{Ust} that
\begin{equation}\label{Psiev}
\Psi_{\text{ev}}(\alpha,\beta;N)=\frac{\log(1+\alpha\beta)}{2\zeta(2)}N^2+O_\epsilon(N^{3/2+\epsilon}),
\end{equation}
where the error term is independent of $\alpha$ and $\beta$. Then if $N=e^{T/2}$, $\Psi_{\text{ev}}(\alpha,\beta;N)$ approximates the number of periodic points $(\omega,-\bar{\omega}^{-1})$ of $\tilde{G}$ such that $\omega\leq\alpha$, $-\bar{\omega}^{-1}\leq\beta$, and $\ell(\omega)\leq T$, to within the error term of $O_\epsilon(N^{3/2+\epsilon})$. This error term takes into account imprecisions due to the small discrepancy between the length of a given point in $Q_G$ and the trace of its corresponding matrix product, the approximations \eqref{CFapprox} which can be dealt with by considering only products with $q_{2m}\geq\sqrt{N}$ and including the remaining products in the error term (see \cite[pp.\ 777--778]{Ust}), and the overcounting of the periodic points in $\Psi_{\text{ev}}(\alpha,\beta;N)$ because of the allowance of products $B^{a_1}A^{a_2}\cdots B^{a_{2m}}$ where $(a_1,\ldots,a_{2m})$ has a minimal even period of less than $2m$ (see the process yielding \eqref{asymptotic1} above for how this is mitigated). Hence, we can use this result to estimate the number of periodic points $(\omega,\tilde{\omega})\in Q_{\tilde{F}}(x,y;T)$ with $\omega\in Q_G$.

To estimate the number of remaining periodic points of $Q_{\tilde{F}}(x,y;T)$, we establish their association, analogous to that discussed in the previous paragraph, with the odd products of $A$ and $B$. Specifically, we associate the odd product in \eqref{mproducts} with $\omega=[a_1,\overline{a_2,\ldots,a_{2m},a_1+a_{2m+1}}]\in Q_F$. Then letting $\bar{\bar{\omega}}:=\tilde{\omega}/(1-\tilde{\omega})=[a_{2m+1},\overline{a_{2m},\ldots,a_2,a_1+a_{2m+1}}]$, we have the following:
\begin{align}
\left|\omega-\frac{p_{2m+1}}{q_{2m+1}}\right|&=\left|[a_1,\ldots,a_{2m+1}+\omega^{-1}]-\frac{p_{2m+1}}{q_{2m+1}}\right|=\left|\frac{p_{2m+1}+\omega^{-1}p_{2m}}{q_{2m+1}+\omega^{-1}q_{2m}}-\frac{p_{2m+1}}{q_{2m+1}}\right|\nonumber\\
&=\frac{1}{q_{2m+1}^2(\omega+q_{2m}/q_{2m+1})}\leq\frac{1}{q_{2m+1}^2(\omega+\bar{\bar{\omega}})},\label{oddapprox1}\\
\left|\bar{\bar{\omega}}-\frac{q_{2m}}{q_{2m+1}}\right|&=\left|[a_{2m+1},\ldots,a_{1}+\bar{\bar{\omega}}^{-1}]-\frac{q_{2m}}{q_{2m+1}}\right|\nonumber\\
&=\left|\frac{q_{2m}+\bar{\bar{\omega}}^{-1}q_{2,2m}}{q_{2m+1}+\bar{\bar{\omega}}^{-1}q_{2,2m+1}}-\frac{q_{2m}}{q_{2m+1}}\right|=\frac{1}{q_{2m+1}^2(\bar{\bar{\omega}}+q_{2,2m+1}/q_{2m+1})}\nonumber\\
&=\frac{1}{q_{2m+1}^2(\bar{\bar{\omega}}+p_{2m+1}/q_{2m+1})}\leq\frac{1}{q_{2m+1}^2(\omega+\bar{\bar{\omega}})}.\label{oddapprox2}
\end{align}
The second bound follows from the equalities
\[\frac{p_{2m+1}(a_{2m+1},\ldots,a_1)}{q_{2m+1}(a_{2m+1},\ldots,a_1)}=\frac{q_{2m}(a_1,\ldots,a_{2m+1})}{q_{2m+1}(a_1,\ldots,a_{2m+1})},\quad
\frac{p_{2m}(a_{2m+1},\ldots,a_1)}{q_{2m}(a_{2m+1},\ldots,a_1)}=\frac{q_{2,2m}(a_1,\ldots,a_{2m+1})}{q_{2,2m+1}(a_1,\ldots,a_{2m+1})},\]
\[p_{2m+1}(a_1,\ldots,a_{2m+1})=q_{2,2m+1}(a_1,\ldots,a_{2m+1}),\]
all of which can be seen from \eqref{CFmatrices}. Also, the trace $q_{2m}+p_{2m+1}$ of the odd product in \eqref{mproducts} is very close to $e^{\ell(a_1+a_{2m+1},a_2,\ldots,a_{2m})/2}$. One can therefore estimate the number of points $(\omega,\tilde{\omega})\in Q_{\tilde{F}}(x,y;T)$ with $\omega\notin Q_G$ by the following lemma.

\begin{lemma}\label{Lemappendix}
For every $\alpha,\beta\in[0,1]$, $(\alpha,\beta)\neq(0,0)$, let $S_N(\alpha,\beta)$ be the cardinality of the set of odd products in \eqref{mproducts} satisfying the inequalities $p_{2m+1}/q_{2m+1}\geq\alpha$, $q_{2m}/q_{2m+1}\geq\beta$, and $q_{2m}+p_{2m+1}\leq N$, which is the set
\begin{equation*}
\left\{ \left( \begin{matrix} q & q^\prime \\ p & p^\prime \end{matrix}\right) \in \SL_2(\Z) :
0\leq p\leq q,\  \alpha q^\prime \leq p^\prime \leq q^\prime,\
\beta q^\prime \leq q\leq q^\prime,\
 p^\prime +q \leq N \right\}.
\end{equation*}
Then for $\epsilon>0$, we have
\begin{equation*}
S_N(\alpha,\beta)= \frac{N^2}{2\zeta(2)}
\log \left( \frac{(1+\alpha)(1+\beta)}{2(\alpha+\beta)} \right) + O_\epsilon \bigg( \Big( \frac{N}{\alpha+\beta}\Big)^{3/2+\epsilon}\bigg) .
\end{equation*}
\end{lemma}

\begin{proof}
We closely follow the proof of \cite[Theorem 2]{Ust}. We can write
\begin{equation*}
S_N(\alpha,\beta)=\sum\limits_{q^\prime \leq \frac{N}{\alpha+\beta}} S_{q^\prime} (\alpha,\beta;N),
\end{equation*}
where
\begin{equation*}
S_{q^\prime} (\alpha,\beta;N) =\sum\limits_{\substack{(p^\prime ,q)\in [\alpha q^\prime,q^\prime]\times [\beta q^\prime,q^\prime] \\
p^\prime q\equiv 1 \hspace{-6pt} \pmod{q^\prime} ,\, p^\prime +q \leq N}} 1
\end{equation*}
can be expressed, employing a standard estimate relying on the Weil bound for Kloosterman sums
(cf. \cite[Lemma 2]{Ust}), as
\begin{equation*}
S_{q^\prime} (\alpha,\beta;N) =\frac{\varphi(q^\prime)}{q^{\prime 2}} \,
\operatorname{Area} (\Omega_{q^\prime} (N,\alpha,\beta)) +O_\epsilon (q^{\prime 1/2+\epsilon}) ,
\end{equation*}
with
\begin{equation*}
\Omega_{q^\prime} (N,\alpha,\beta) = ( [\alpha q^\prime ,q^\prime ] \times [\beta q^\prime ,q^\prime ])
\cap \{ (u,v): u+v \leq N\} .
\end{equation*}

Assuming without loss of generality that $0\leq\beta \leq \alpha\leq 1$, we plainly compute
\begin{equation*}
\operatorname{Area} (\Omega_{q^\prime} (N,\alpha,\beta))=\begin{cases}
(1-\alpha)(1-\beta) q^{\prime 2} & \mbox{\rm if $q^\prime \leq \frac{N}{2}$} \\
(1-\alpha)(1-\beta)q^{\prime 2} -\frac{1}{2} (2q^\prime -N)^2 & \mbox{\rm if $\frac{N}{2} \leq q^\prime \leq \frac{N}{1+\alpha}$}  \\
\frac{(1-\alpha)q^\prime}{2} (2N-(1+\alpha+2\beta)q^\prime)  &
\mbox{\rm if $\frac{N}{1+\alpha} \leq q^\prime \leq \frac{N}{1+\beta}$} \\
\frac{1}{2} (N-(\alpha+\beta)q^\prime)^2 & \mbox{\rm if $\frac{N}{1+\beta} \leq q^\prime \leq \frac{N}{\alpha+\beta}$} \\
0 & \mbox{\rm if $\frac{N}{\alpha+\beta}\leq q^\prime$.}
\end{cases}
\end{equation*}
Employing in the sequel the estimates
\begin{equation*}
\begin{split}
\sum_{n\leq X} \varphi (n) & =\frac{X^2}{2\zeta(2)} +O(X\log X) ,\qquad
\sum_{n\leq X} \frac{\varphi(n)}{n} =\frac{X}{\zeta(2)} +O(\log X) ,\\
\sum_{n\leq X} \frac{\varphi(n)}{n^2} & =\frac{\log X}{\zeta(2)} +C +O\bigg( \frac{\log X}{X}\bigg)
\quad \mbox{\rm with} \quad C=\frac{1}{\zeta(2)} \bigg( \gamma -\frac{\zeta^\prime (2)}{\zeta(2)} \bigg),
\end{split}
\end{equation*}
where the first two are elementary, while the third one is proved using properties of the Riemann zeta function and Perron's formula
(see, e.g. \cite[Corollary 4]{Bo} or \cite[page 32]{Te}), we find that the main term contributions of terms of the form $\varphi(q^\prime)/q^\prime$
and $\varphi(q^\prime)$ to $S_N(\alpha,\beta)$ are both $\ll_\epsilon\big( \frac{N}{\alpha+\beta}\big)^{3/2+\epsilon}$, while the contribution
of terms of the form $\varphi(q^\prime)/q^{\prime 2}$ is given by the main term $\frac{N^2}{2\zeta(2)}
\log \big( \frac{(1+\alpha)(1+\beta)}{2(\alpha+\beta)} \big)$ and an error which is again $\ll_\epsilon\big(\frac{N}{\alpha+\beta}\big)^{3/2+\epsilon}$.
\end{proof}

Analogously to $\Psi_{\text{ev}}$, when $N=e^{T/2}$, $S_N(\alpha,\beta)$ estimates the number of periodic points in $(\omega,\tilde{\omega})$ of $\tilde{F}$ such that $\omega\geq\alpha$, $\bar{\bar{\omega}}\geq\beta$, and $\ell(\omega)\leq T$. Errors arise for the same reasons, and can be dealt with in essentially the same way, as the $\Psi_{\text{ev}}$ approximation. In particular, when analyzing the error due to the approximations \eqref{oddapprox1} and \eqref{oddapprox2}, one can focus attention on products where $q_{2m+1}\geq\big(\frac{N}{\alpha+\beta}\big)^{1/2}$ and include the remaining products in the error term, analogously to \cite[pp.\ 777--778]{Ust}. Additionally, if one chooses to shrink $\alpha+\beta$ to $0$ as $N\rar\infty$, one should impose a restriction $\alpha+\beta\gg N^{-1/3+\delta_\epsilon}$ ($\delta_\epsilon>8\epsilon/(9+6\epsilon)$) to obtain a main term, and also ensure that the bounds \eqref{oddapprox1} and \eqref{oddapprox2} are sufficiently small compared to $\alpha+\beta$. Under this condition, the overall error in the estimate is $\ll_\epsilon\big(\frac{N}{\alpha+\beta}\big)^{3/2+\epsilon}$.

We can now complete the proof of Theorem \ref{T1appendix}. Assume henceforth that $N=e^{T/2}$. We first consider the case where $y\geq1/2$. Then the only points $(\omega,\tilde{\omega})$ in
$Q_{\tilde{F}} (x,y;T)$ are such that $\omega\in Q_G$, for if $\omega\notin Q_G$, then $\tilde{\omega}<1/2$. So we simply have to estimate
$\# \{ (\omega,\tilde{\omega} ): \omega \in Q_G (T), \omega \geq x, \tilde{\omega} \geq y\}$.  For
$\omega=[\overline{a_1,\ldots,a_n}] \in Q_G(T)$, $\tilde{\omega}=[1,\overline{a_n,\ldots,a_1}] = (1-\bar{\omega}^{-1})^{-1}$, and thus $\tilde{\omega}\geq y$ is
equivalent to $-\bar{\omega}^{-1} \leq y^{-1}-1$. Employing \eqref{Psiev}, we can infer that
\begin{equation*}
\begin{split}
\#Q_{\tilde{F}} (x,y;T) & = \Psi_{\text{ev}}( 1,y^{-1}-1;N) -
\Psi_{\text{ev}} ( x, y^{-1}-1;N)+O_\epsilon(N^{3/2+\epsilon}) \\
& = \frac{N^2}{2\zeta(2)} \log \left( \frac{1+(y^{-1}-1)}{1+x(y^{-1}-1)} \right)
+O_\epsilon (N^{3/2+\epsilon})\\ &  =
\frac{N^2}{2\zeta(2)} \log \left( \frac{1}{x+y-xy}\right) +O_\epsilon (N^{3/2+\epsilon}).
\end{split}
\end{equation*}

Assume next that $y<1/2$. Then $y^{-1}-1>1$; so again employing \eqref{Psiev}, we see that the contribution to $\#Q_{\tilde{F}}(x,y;T)$ of the periodic points $(\omega,\tilde{\omega})$ with $\omega\in Q_G$ is
\begin{equation}
\Psi_{\text{ev}}(1,1;N)-\Psi_{\text{ev}}(x,1;N) +O(N^{3/2+\epsilon})
=\frac{N^2}{2\zeta(2)} \log \bigg( \frac{2}{1+x}\bigg) +O_\epsilon (N^{3/2+\epsilon}).\label{QGcontribution}
\end{equation}
On the other hand, since $\tilde{\omega}\geq y$ is equivalent to $\bar{\bar{\omega}}\geq y/(1-y)$, the contribution to $\#Q_{\tilde{F}}(x,y;T)$ of the points $(\omega,\tilde{\omega})$ with $\omega\notin Q_G$ is,
according to Lemma \ref{Lemappendix} with $\alpha=x$ and $\beta=y/(1-y)$,
\begin{align}\label{nonQGcontribution}
&S_N\left(x,\frac{y}{1-y}\right)+O_\epsilon\left(\left(\frac{N}{\alpha+\beta}\right)^{3/2+\epsilon}\right)
=\frac{N^2}{2\zeta(2)}
\log \left( \frac{(1+x)(1+\frac{y}{1-y})}{2(x+\frac{y}{1-y})} \right) + O_\epsilon \left( \left( \frac{N}{x+y}\right)^{3/2+\epsilon}\right)\nonumber\\
&\qquad\qquad\qquad\qquad\qquad\qquad\qquad
=\frac{N^2}{2\zeta(2)}
\log \left( \frac{(1+x)}{2(x+y-xy)} \right) + O_\epsilon \left( \left( \frac{N}{x+y}\right)^{3/2+\epsilon}\right).
\end{align}
The statement of Theorem \ref{T1appendix} follows by adding the quantities in \eqref{QGcontribution} and \eqref{nonQGcontribution}. Lastly, we note that $x+y$ can shrink to $0$ as $N\rar\infty$ as long as $x+y\gg N^{-1/3+\delta_\epsilon}$.

\end{document}